\documentclass{article} 

\def\TEXPAD{0}

\usepackage{amsmath,amssymb,amsthm}
\usepackage[mathscr]{eucal}
\usepackage{mathtools}
\usepackage{fancyhdr}
\usepackage{cite}
\usepackage{hyperref}
\usepackage[capitalize, nameinlink]{cleveref}
\usepackage{enumitem}
\usepackage[usenames,dvipsnames]{xcolor}
\usepackage{tikz}
\usepackage[font=small,labelfont=bf, margin=.5in]{caption}

\colorlet{prettygreen}{ForestGreen!60!LimeGreen}

\usetikzlibrary{calc}
\tikzset{vtx/.style={circle, fill, inner sep=1.5pt}}
\tikzset{openvtx/.style={circle, draw, inner sep=1.5pt}}

\makeatletter
\newtheorem*{rep@theorem}{\rep@title}
\newcommand{\newreptheorem}[2]{%
\newenvironment{rep#1}[1]{%
 \def\rep@title{#2~\ref{##1}}%
 \begin{rep@theorem}}%
 {\end{rep@theorem}}}
\makeatother

\newtheorem{theorem}{Theorem}[section]
\newreptheorem{theorem}{Theorem}
\newtheorem{thm}[theorem]{Theorem}
\newtheorem{lemma}[theorem]{Lemma}
\newtheorem{prop}[theorem]{Proposition}
\newtheorem{corollary}[theorem]{Corollary}

\newtheorem*{claim*}{Claim}

\theoremstyle{definition}
\newtheorem{dfn}[theorem]{Definition}

\newtheorem{question}[theorem]{Question}
\newtheorem{conjecture}[theorem]{Conjecture}
\newtheorem{example}[theorem]{Example}

\theoremstyle{remark}
\newtheorem{rmk}[theorem]{Remark}

\newlist{thmenum}{enumerate}{3}
\setlist[thmenum,1]{label={(\roman*)}, ref={\thetheorem(\roman*)}}
\newlist{homtenum}{enumerate}{1}
\setlist[homtenum,1]{leftmargin=36pt}

\if\TEXPAD0
	\crefname{claim}{Claim}{Claims}
	\crefname{prop}{Proposition}{Propositions}
\fi

\newcommand{\I}{{\mathscr I}}

\newcommand{\Co}{{\mathfrak C}}

\newcommand{\EE}{{\mathcal E}}
\newcommand{\NN}{{\mathcal N}}

\newcommand{\al}{\alpha}

\newcommand{\del}{\delta}
\newcommand{\eps}{\epsilon}
\DeclareMathOperator{\BG}{BG}
\DeclareMathOperator{\cyc}{cyc}
\newcommand{\ccsim}{\overset{\diamond}{\sim}}
\newcommand{\ccl}[2][]{[#2]^\diamond_{#1}}
\newcommand{\dhom}{\delta_{\hom}}
\newcommand{\dchi}{\delta_{\chi}}
\newcommand{\dcyc}{\delta_{\cyc}}
\newcommand{\pieven}{\pi_1^\textsc{even}}
\DeclareMathOperator{\sd}{\bigtriangleup}
\newcommand{\td}{\widetilde}
\newcommand{\multiset}[1]{\left\lbrace \! \left\lbrace #1 \right\rbrace \! \right\rbrace}

\newcommand{\subarrow}{\ \xrightarrow{\text{\ sub\ }}\ }
\newcommand{\insarrow}{\xrightarrow{\text{\ ins\ }}}
\newcommand{\delarrow}{\xrightarrow{\text{\ del\ }}}
\newcommand{\homtarrow}[1]{\xrightarrow{\text{\ #1\ }}}

\title{Homotopy and the Homomorphism Threshold of Odd Cycles}
\author{Maya Sankar\thanks{Department of Mathematics, Stanford University, Stanford, CA 94305. Email: {\tt mayars@stanford.edu}. Research supported by NSF GRFP Grant DGE-1656518 and a Hertz Fellowship.}}
\date{\today}
\makeatletter
\let\mytitle\@title
\let\myauthor\@author
\makeatother
\begin{document}
\maketitle

\begin{abstract}
	Consider a family $\mathcal F$ of $C_{2r+1}$-free graphs, where $r\geq 2$. Suppose that each graph in $\mathcal F$ has minimum degree linear in its number of vertices. Thomassen showed that such a family has bounded chromatic number, or, equivalently, that all graphs in $\mathcal F$ are homomorphic to a complete graph of bounded size. Considering instead homomorphic images which are themselves $C_{2r+1}$-free, we construct a family of dense $C_{2r+1}$-free graphs with no $C_{2r+1}$-free homomorphic image of bounded size. This provides the first nontrivial lower bound on the homomorphism threshold of odd cycles of length at least 5 and answers a question of Ebsen and Schacht.

	Our proof introduces a new technique to describe the topological structure of a graph. We establish a graph-theoretic analogue of homotopy equivalence, which allows us to analyze the relative placement of odd closed walks in a graph. This notion has unexpected connections to the neighborhood complex, leading to multiple interesting questions.
\end{abstract}

\section{Introduction}

A central question in graph theory is to understand those graphs containing no copy of some fixed subgraph $H$. In one of the first applications of the probabilistic method, Erd\H{o}s \cite{Er59} showed that there are $H$-free graphs of arbitrarily large chromatic number (so long as $H$ is not a forest). However, Erd\H{o}s's construction is inherently sparse.

A natural follow-up is to restrict one's focus to $H$-free graphs that are dense in some way, such as having large minimum degree. This interplay of chromatic number and minimum degree was first investigated in the case that  $H$ is a clique, almost fifty years ago. Given a graph $G$, let $|G|$ denote its number of vertices. Andr\'asfai, Erd\H os, and S\'os \cite{AnErSo74} showed that any $K_r$-free graph $G$ with minimum degree $\delta(G)> \frac{3r-7}{3r-4}|G|$ has chromatic number $\chi(G)< r$. Almost simultaneously, Erd\H os and Simonovits \cite{ErSi73} constructed a family of triangle-free graphs $G_n$ with unbounded chromatic number that satisfy $\delta(G_n)/|G_n|=\frac 13-o(1)$.

The results of \cite{ErSi73,AnErSo74} inaugurated the study of the chromatic threshold and, later, the homomorphism threshold. Both investigate the dependency of structural properties, such as the chromatic number, on the minimum degree of a dense $H$-free graph.

\subsection{Thresholds}\label{s:thresholds}

The chromatic and homomorphism thresholds provide different formalizations of the following question. Let $G$ be an $H$-free graph on $n$ vertices, with minimum degree at least $\alpha n$. For which fixed values of $\alpha$ can we provide guarantees about the structure of $G$?

The chromatic threshold asks for the infimum of those $\al$ for which $\chi(G)$ is bounded as a function of $\al$. More formally, the \emph{chromatic threshold} of a graph $H$ is the quantity
\begin{align*}
\dchi(H)=\inf\{\al\colon\exists\ &K=K(H,\al)\text{ such that}
\\&\text{any $H$-free graph $G$ with $\delta(G)\geq \al|G|$ satisfies $\chi(G)\leq K$}
\}.
\end{align*}
We remark that the chromatic threshold of any bipartite graph is 0, by the K\H{o}v\'ari--S\'{o}s--Tur\'{a}n theorem \cite{KoSoTu54}, and that by the Erd\H{o}s--Stone theorem \cite{ErSt46}, any $H$ with $\chi(H)=r\geq 3$ has chromatic threshold at most the Tur\'an density $\frac{r-2}{r-1}$ of $H$.

Aside from bipartite graphs, the first chromatic thresholds computed were those of cliques and odd cycles. Thomassen \cite{Th02} showed that $\dchi(K_3)=\frac 13$, matching the lower bound from \cite{ErSi73}. This was later extended by Goddard and Lyle \cite{GoLy11} and Nikiforov \cite{Ni10}, who computed $\dchi(K_r)$ for all $r\geq 3$. Thomassen \cite{Th07} additionally gave a very short proof that $\dchi(C_{2r+1})=0$ for all $r\geq 2$. Then, in 2010, \L uczak and Thomass\'e \cite{LuTh10} and Lyle \cite{Ly11} proved bounds on the chromatic thresholds of larger families of graphs. Building on these results, Allen, B\"ottcher, Griffiths, Kohayakawa, and Morris \cite{Aletal13} completely determined the chromatic threshold of any graph $H$.

The triangle-free graphs with $\del(G)>\frac 13|G|$ may be described more precisely in terms of graph homomorphisms. If $G$ and $\Gamma$ are graphs, a \emph{homomorphism} $\phi:G\to \Gamma$ is a map $V(G)\to V(\Gamma)$ such that for any edge $xy\in E(G)$, we have $\phi(xy):=\phi(x)\phi(y)\in E(\Gamma)$. If such a $\phi$ exists, we say $\Gamma$ is a \emph{homomorphic image} of $G$; similarly, if $\mathcal G$ is a family of graphs and $\Gamma$ is a homomorphic image of each $G\in\mathcal G$, we say $\Gamma$ is a \emph{homomorphic image} of $\mathcal G$.

 In 2006, \L uczak \cite{Lu06} showed that for every $\eps>0$, there is a finite triangle-free graph $\Gamma_\eps$ such that every triangle-free graph $G$ with minimum degree at least $(\frac 13+\eps)|G|$ has a homomorphism to $\Gamma_\eps$. Later, Brandt and Thomass\'e \cite{BrTh11} fully described the minimal such graphs $\Gamma_\eps$, showing that they are four-colorable and have $O\left(\frac 1\eps\right)$ vertices. This motivated the introduction of the \emph{homomorphism threshold} of a graph $H$, defined as
\begin{align*}
\dhom(H)=\inf\{\al\colon\exists\ &\text{an $H$-free graph $\Gamma=\Gamma(H,\al)$ such that}
\\&\text{any $H$-free graph $G$ with $\delta(G)\geq\al|G|$ has a homomorphism to $\Gamma$}
\}.
\end{align*}

We remark that $\dhom(H)\geq\dchi(H)$ for any graph $H$, because if there is a homomorphism $G\to\Gamma$ then $\chi(G)\leq\chi(\Gamma)$. \L uczak's work showed that this inequality is tight for $K_3$, i.e.\ that $\dhom(K_3)=\frac 13=\dchi(K_3)$. Goddard and Lyle \cite{GoLy11} and Oberkampf and Schacht \cite{ObSc20} extended this result to all cliques, showing that $\dhom(K_r)=\dchi(K_r)$ for any $r\geq 4$. There has also recently been some work on the homomorphism thresholds of odd cycles. Let $\mathcal C_{2r+1}=\{C_3,C_5,\ldots,C_{2r+1}\}$; the homomorphism threshold of a family $\mathcal H$ of forbidden subgraphs is defined analogously to that of a single graph $H$. It was shown that $\dhom(\mathcal C_{2r+1})=\frac{1}{2r+1}$ and $\dhom(C_{2r+1})\leq\frac 1{2r+1}$ by Letzter and Snyder \cite{LeSn19} (for $r=2$) and Ebsen and Schacht \cite{EbSc20} (for $r\geq 2$).

Computing the homomorphism thresholds of graphs other than cliques and bipartite graphs remains a challenging problem. An inherently related question is whether the inequality $\dhom(H)\geq\dchi(H)$ is ever strict. The results of \cite{LeSn19,EbSc20} partially resolved this, showing that the families $\mathcal H=\mathcal C_{2r+1}$ satisfy $\dhom(\mathcal H)>\dchi(\mathcal H)$ for $r\geq 2$. However, it remained unknown whether there were single graphs $H$ for which $\dhom(H)>\dchi(H)$, and Ebsen and Schacht \cite{EbSc20} explicitly asked whether this holds for $H=C_5$.

Our work answers their question, presenting the first nontrivial lower bound on $\dhom(C_{2r+1})$ for $r\geq 2$. 
\begin{thm}\label{thm:dhom>0}
For any $r\geq 2$, we have $\dhom(C_{2r+1})>0$. 
\end{thm}
\noindent Recall that Thomassen \cite{Th07} showed $\dchi(C_{2r+1})=0$ for $r\geq 2$. Thus, these odd cycles are the first graphs $H$ known to satisfy $\dhom(H)>\dchi(H)$. An explicit lower bound on $\dhom(C_{2r+1})$ is stated in terms of the volume of a spherical cap in \cref{cor:bound-dhom}.

We remark that the results of \cite{Lu06,GoLy11,ObSc20,LeSn19,EbSc20} focus on providing upper bounds for $\dhom$. The lower bound $\dhom(\mathcal C_{2r+1})\geq\frac 1{2r+1}$ is achieved without difficulty via a sequence $G_1,G_2,\ldots$ of increasingly good approximations to a Borsuk graph (defined in \cref{s:borsukintro} below), using arguments akin to those in \cref{s:borsuk}. One could alternatively choose the approximations $G_i$ a little more carefully, as in \cite{EbSc20}, so that any two vertices lie on a common $C_{2r+3}$ --- immediately implying that any proper homomorphic image of $G_i$ contains a shorter odd cycle.

Our work marks the first more elaborate technique for lower-bounding homomorphism thresholds. We introduce a new tool to describe the topological structure of a graph. Then, letting $G$ be a graph with appropriate topological properties, we reduce the study of $C_{2r+1}$-free homomorphic images of $G$ to the more tractable analysis of certain small $C_{2r+1}$-free graphs. For some families $\mathcal G$ of dense $C_{2r+1}$-free graphs, these techniques show that all $C_{2r+1}$-free homomorphic images of any $G\in\mathcal G$ have size at least some increasing function of $|G|$, thereby proving \cref{thm:dhom>0}.

\subsection{The Borsuk Graph and Topological Methods}\label{s:borsukintro}

Our topological methods are inspired by Borsuk graphs, infinite graphs constructed on the $n$-dimensional unit sphere $S^n$. These were first introduced by Erd\H os and Hajnal \cite{ErHa67}, and their properties have since been studied in many different contexts \cite{KaMa20,Lo83,SiTa11}. We seek to introduce graph properties that, when applied to Borsuk graphs, capture the underlying topological structure of $S^n$.

Given $n\geq 1$ and $\eps>0$, the Borsuk graph $\BG(n,\eps)$ is an infinite graph with vertex set $S^n$. We place an edge between vertices $x,y\in S^n$ if the distance between them on the sphere is greater than $\pi-\eps$. Equivalently, the neighborhood $N(x)$ of any vertex $x\in S^n$ is a ball of radius $\eps$ around its antipode $-x$. 

The Borsuk graph appears frequently in lower-bound constructions for chromatic and homomorphism thresholds and other extremal quantities \cite{Aletal13,BoEr76,EbSc20,LeSn19,LuTh10}. In addition to being dense, its useful properties include having large chromatic number and lacking short odd cycles. More specifically, $\BG(n,\eps)$ has no odd cycles of length at most $1/\eps$ and has chromatic number at least $n+2$; it is well-known (see \cite{ErHa67,Lo83}) that the latter attribute is equivalent to the Borsuk--Ulam theorem.

The Borsuk graph also inspired Lov\'asz's groundbreaking proof of the Kneser conjecture \cite{Lo78}, one of the first applications of topological methods to combinatorics. Lov\'asz introduced the \emph{neighborhood complex} $\NN(G)$ of a graph $G$, a simplicial complex which, when $G=\BG(n,\eps)$, captures some of the topological structure of $S^n$. He then bounded $\chi(G)$ in terms of topological properties of $\NN(G)$, as follows. We say a topological space $X$ is \emph{$k$-connected} if for each $\ell=0,1,\ldots,k$, any continuous map $S^\ell\to X$ extends continuously to a map from the entire ball $B^{\ell+1}$.

\begin{thm}[Lov\'asz \cite{Lo78}]\label{thm:lovasz}
Let $G$ be a graph whose neighborhood complex $\NN(G)$ is $k$-connected. Then $\chi(G)\geq k+3$.
\end{thm}
\noindent Lov\'asz's proof of the Kneser conjecture concluded by showing that the neighborhood complex of a Kneser graph had the desired connectivity.

Our work introduces a new tool to describe the topological structure of a graph. We introduce an equivalence relation on the set of walks in a connected graph $G$, which we call \emph{(discrete) homotopy equivalence} in analogy to the algebraic topology setting. This detects topological properties similar to those of the neighborhood complex $\NN(G)$. One may define a \emph{discrete fundamental group} $\pi_1(G)$, consisting of homotopy classes of those closed walks based at some fixed $v\in V(G)$. When $G$ is nonbipartite --- a condition ensuring that $\NN(G)$ is connected --- we prove (see \cref{prop:pi1-NG}) that the fundamental group of the topological space $\NN(G)$ is isomorphic to the index-2 subgroup $\pieven(G)\subseteq\pi_1(G)$ comprising closed walks of even length.

Although discrete homotopy and the neighborhood complex are interrelated, the former is distinctly better suited for our application. Because $\pi_1(\NN(G))\simeq\pieven(G)$, studying $\NN(G)$ inherently limits our perspective to even closed walks, whereas analyzing interactions between odd cycles is our goal. Additionally, due to its purely combinatorial definition, discrete homotopy can be simpler to reason about.

At the heart of our argument is the following theorem, from which we deduce \cref{thm:dhom>0}. We say a connected graph $G$ is \emph{simply connected} (see \cref{dfn:simp-conn}) if walks in $G$ under discrete homotopy behave like continuous paths in a simply connected topological space under homotopy. An equivalent formulation, discussed in \cref{s:topology}, is that $G$ is simply connected if and only if $\pieven(G)$ is trivial.
\begin{theorem}\label{mainthm}
Let $G$ be a simply connected graph and fix $r\geq 2$. Suppose $\phi:G\to H$ is a graph homomorphism such that $H$ is $C_{2r+1}$-free, and suppose further that some odd cycle $C$ of $G$ satisfies $|\phi(V(C))|\leq 2r+2$. Then $\chi(G)<8r^2$.
\end{theorem}

Consider the $C_{2r+1}$-free Borsuk graph $\BG\left(8r^2-2,\frac{\pi}{2r+1}\right)$, which has chromatic number at least $8r^2$. Applying \cref{mainthm}, we show that any family of arbitrarily good finite approximations to this graph has no $C_{2r+1}$-free homomorphic image of bounded size. This allows us to lower-bound $\dhom(C_{2r+1})$, proving \cref{thm:dhom>0}.

\paragraph*{Organization.} In \cref{s:homotopy}, we introduce discrete homotopy and other related preliminaries. In \cref{s:mainthm-reduction}, we reduce the proof of \cref{mainthm} to studying a well-connected subgraph of the $C_{2r+1}$-free graph $H$; we then complete the proof of \cref{mainthm} in \cref{s:mainthm-pf}. We deduce \cref{thm:dhom>0} from \cref{mainthm} in \cref{s:borsuk} by analyzing approximations to Borsuk graphs. \cref{s:topology} builds out our topological framework, defining the discrete fundamental group and proving that $\pieven(G)\simeq\pi_1(\NN(G))$ when $G$ is non-bipartite. Lastly, we conclude by discussing potential extensions of \cref{mainthm} and some topological open problems in \cref{s:conclusion}.

\paragraph*{Notation.} If $G$ is a graph, we write $|G|$ for its number of vertices and, for any $v\in V(G)$, denote its neighborhood by $N(v)$.
In \cref{s:homotopy,s:mainthm-reduction}, we frequently make use of multisets. We use double braces $\multiset\cdot$ to denote a multiset, and, to further aid the reader, we establish the convention that multisets are always denoted by variables with tildes. We write $|\td A|$ for the cardinality of a multiset (taken with multiplicity) and write $\td A_1+\td A_2$ for the sum of two multisets. Additionally, if $\td A$ is a multiset and $B$ a set, let $\td A\cap B$ denote the multiset containing each $b\in B$ with multiplicity equal to its multiplicity in $\td A$.

\section{Homotopy and Invariants}\label{s:homotopy}

In this section, we define an equivalence relation, homotopy equivalence, on walks in a graph $G$. We also introduce a family of functions --- called intersection invariants --- on walks in $G$, and prove they are homotopy-invariant.

Let $G$ be a graph. A \emph{walk} of length $k\geq 0$ in $G$ is a sequence $P=v_0\cdots v_k$ of $k+1$ vertices such that $v_iv_{i+1}$ is an edge of $G$ for each $0\leq i < k$. We say $P$ is a \emph{closed walk} if its endpoints $v_0$ and $v_k$ are equal. We describe $P$ as \emph{odd} or \emph{even} depending on the parity of its length $k$, and abbreviate this as the \emph{parity} of the walk $P$. We write $\td E(P)$ for the multiset of edges of $P$ counted with multiplicity; that is,
\[
\td E(P)=\multiset{v_iv_{i+1}:0\leq i<k}.
\]

There are two key differences between walks and paths. Unlike paths, walks may repeat vertices. Moreover, walks are \emph{oriented} --- in general, $v_0\cdots v_k$ and $v_k\cdots v_0$ are different walks. The same distinctions hold between closed walks and cycles.

In a topological space, two continuous paths are homotopic if one may be continuously deformed into the other. We introduce a discrete analogue of this concept, in which two walks of a graph are homotopic if one may be transformed into the other via a sequence of small perturbations.

\begin{dfn}\label{dfn:homotopy}
Let $G$ be a graph. We say two walks $P$ and $P'$ of $G$ are \emph{homotopic} or \emph{homotopy equivalent} if we may transform $P$ into $P'$ via a finite sequence of substitution, insertion, and deletion steps, defined as follows.
\begin{homtenum}
\item[(Sub)] Given a walk $v_0\cdots v_k$ and an index $0<i<k$, replace $v_i$ with $v'_i$ for some $v'_i\in N(v_{i-1})\cap N(v_{i+1})$. 
\item[(Ins)] Given a walk $v_0\cdots v_k$ and an index $0\leq i\leq k$, replace $v_i$ with $v_iwv_i$ for some $w\in N(v_i)$.
\item[(Del)] Given a walk $v_0\cdots v_k$ and an index $0<i<k$ satisfying $v_{i-1}=v_{i+1}$, replace $v_{i-1}v_iv_{i+1}$ with $v_{i-1}$.
\end{homtenum}
\end{dfn}

Notice that these operations are invertible: we can undo a substitution step via another substitution step, and similarly insertions and deletions are each other's inverses. Thus, homotopy does give an equivalence relation on walks in $G$.

\newcommand{\x}[1]{{\color{red}#1}}
\begin{example}
Let $G$ be the 7-vertex graph shown below. Then the two closed walks $v_0v_1v_2v_0$ and $v_0v_5v_6v_0$ are homotopic, via the sequence of operations given.
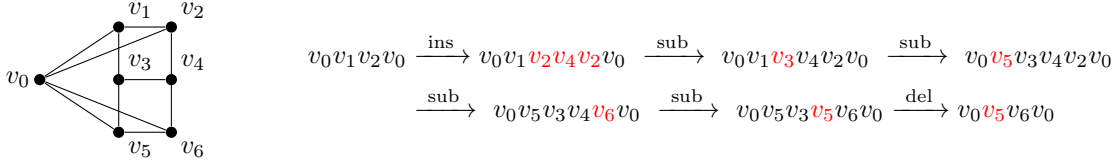
\[\begin{tikzpicture}[scale=0.7]
	\foreach \x/\y [count=\i from 0]
		in {-0.5/0,1/1,2/1,1/0,2/0,1/-1,2/-1}
		\coordinate[vtx](V\i) at (\x,\y);
	\foreach \pos [count=\i from 0]
		in {left, above right, above right, above right, above right, below right, below right}
		\node at (V\i) [\pos] {$v_\i$};
	\draw (V1) -- (V5) -- (V6) -- (V2) -- (V1);
	\draw (V3) -- (V4);
	\foreach \i in {1,2,5,6}
		\draw (V0) -- (V\i);
		
	\draw (4,0) node[right] {
	$\begin{array}{l}
		\ v_0v_1v_2v_0
		\insarrow v_0v_1\x{v_2v_4v_2}v_0
		\subarrow v_0v_1\x{v_3}v_4v_2v_0
		\subarrow v_0\x{v_5}v_3v_4v_2v_0
		\\[5pt]\hphantom{v_0v_1v_2v_0}\subarrow v_0v_5v_3v_4\x{v_6}v_0
		\subarrow v_0v_5v_3\x{v_5}v_6v_0
		\delarrow v_0\x{v_5}v_6v_0
	\end{array}$
	};
		
\end{tikzpicture}\]
\end{example}

We state two homotopy invariants which are immediately clear from the definition.

\begin{prop}\label{prop:homt-trivial}
Let $G$ be a graph and suppose that the walks $P=v_0\cdots v_k$ and $P'=w_0\cdots w_\ell$ are homotopic in $G$. Then,
\begin{thmenum}
\item $P$ and $P'$ have the same endpoints, i.e.\ $v_0=w_0$ and $v_k=w_\ell$.
\item We have that $k\equiv\ell$ (mod 2).\label[prop]{prop:homt-trivial-parity}
\end{thmenum}
\end{prop}

To design more informative homotopy invariants, we need some additional notation. Revisiting \cref{dfn:homotopy}, notice that concatenating the length-2 paths removed and added in (Sub) forms a closed walk $v_{i-1}v_iv_{i+1}v'_iv_{i-1}$ of length 4. This motivates studying how the edges of a graph connect via $C_4$'s.

Let $H$ be a graph. Let $\ccsim$ be the equivalence relation on $E(H)$ generated by the relations
	\[\{e\ccsim e'\colon e,e'\in E(H)\text{ contained in the same $C_4$}\}.
	\]
For any edge $e\in E(H)$, we define its \emph{$C_4$-closure}, written $\ccl[H]e$, to be the equivalence class of $e$ under $\ccsim$.

Now, fix a graph homomorphism $\phi:G\to H$. Given $e\in E(G)$ and its image $f=\phi(e)\in E(H)$, let $G'$ be the subgraph of $G$ induced by the edge set $\phi^{-1}\left(\ccl[H]f\right)$. Define $\ccl[\phi]e$ to be the edge set of the connected component of $G'$ containing $e$. We remark that the sets $\left\{\ccl[\phi]e:e\in E(G)\right\}$ partition $E(G)$.

Say a set $A\subseteq E(G)$ is \emph{$\phi$-stable} if it may be expressed as a union of sets of the form $\ccl[\phi]e$. Equivalently, $A$ is $\phi$-stable if for any edge $e\in A$, we have $\ccl[\phi]e\subseteq A$. Given a $\phi$-stable set $A$ and a vertex $u\in V(H)$, we also define
	\[
	A_u:=\left\{e\in A:e\text{ has an endpoint in }\phi^{-1}(u)\right\}.
	\]

Using this notation, we define our homotopy invariants.

\begin{dfn}\label{dfn:inv}
Fix a graph homomorphism $\phi:G\to H$. The \emph{intersection invariants} associated to $\phi$ are families of functions $\I_A$ and $\I_{A,u}$ for each $\phi$-stable subset $A\subseteq E(G)$ and each vertex $u\in V(H)$. Both $\I_A$ and $\I_{A,u}$ map multisets of edges of $E(G)$ to $\mathbb Z/2\mathbb Z$, and are defined as
\[\I_A(\td F):=|\td F\cap A|\pmod 2\quad\text{ and }\quad\I_{A,u}(\td F):=|\td F\cap A_u|\pmod 2,\]
recalling that the cardinalities of the multisets $\td F\cap A$ and $\td F\cap A_u$ are taken with multiplicity.
When we want to specify the $\phi$-stable set $A$, we refer to the \emph{$A$-intersection invariants} $\I_A$ and $\I_{A,u}$.

\end{dfn}

Given an intersection invariant $\I$ associated to $\phi:G\to H$ and a walk $P$ of $G$, we define $\I(P):=\I(\td E(P))$. We will now show that intersection invariants are in fact homotopy invariants --- if $P$ and $P'$ are homotopic walks in $G$, then $\I(P)=\I(P')$.

\begin{prop}\label{prop:hom-inv-helper}
Fix a graph homomorphism $\phi:G\to H$. Let $\I$ be an intersection invariant associated to $\phi$ of either form $\I_A$ or $\I_{A,u}$.
\begin{thmenum}
\item Suppose $\td E_1$ and $\td E_2$ are multisets of edges of $E(G)$. Then $\I(\td E_1+\td E_2)=\I(\td E_1)+\I(\td E_2)$.\label[prop]{prop:sum-inv}
\item Let $P=v_0v_1v_0$ be a closed walk of length 2 in $G$. Then $\I(P)=0$.
\item Let $P=v_0v_1v_2v_3v_0$ be a closed walk of length 4 in $G$. Then $\I(P)=0$.
\end{thmenum}
\end{prop}

\begin{proof}
Recall from \cref{dfn:inv} that $\I$ is defined as $\I(\td F)\equiv |\td F\cap A'|\pmod 2$ for some $A'\subseteq E(G)$. Then (i) immediately follows from the identity
\[
|(\td E_1+\td E_2)\cap A'|=|\td E_1\cap A'|+|\td E_2\cap A'|.
\]
For (ii), observe $\td E(P)=\multiset{v_0v_1,v_1v_0}$ repeats a single edge twice. Thus $|\td E(P)\cap A'|$ is either 0 or 2, depending on whether $v_0v_1\in A'$, yielding $\I(P)=0$.

It remains to show (iii). Fix a $\phi$-stable set $A\subseteq E(G)$ and suppose that $\I$ is an $A$-intersection invariant. For $i=0,\ldots,3$, set $u_i=\phi(v_i)$. We will use casework on whether or not the four vertices $u_0,\ldots,u_3$ are distinct.

First, suppose that $u_0,\ldots,u_3$ are distinct vertices of $H$, so that $u_0u_1u_2u_3$ is a $C_4$. Then, $u_0u_1\ccsim u_1u_2\ccsim u_2u_3\ccsim u_3u_0$ in $H$, so the four edges of $\td E(P)$ are in the same connected component of $\phi^{-1}\left(\ccl[H]{u_0u_1}\right)$. It follows that $\td E(P)\cap A$ is either $\td E(P)$ or empty, and thus that $\I_A(P)=0$. Now fix a vertex $u\in V(H)$; we will show $\I_{A,u}(P)=0$. We observe that $\td E(P)\cap A_u$ is only nonempty if $u\in\{u_0,\ldots,u_3\}$, in which case exactly two edges of $\td E(P)$ will have an endpoint in $\phi^{-1}(u)$. It follows that $|E(P)\cap A_u|$ is either 0 or 2, so $\I_{A,u}(P)=0$. Thus, $\I(P)=0$ if $u_0,\ldots,u_3$ are distinct, whether $\I$ has the form $\I_A$ or $\I_{A,u}$.

Lastly, we show (iii) in the case that the vertices $u_0,\ldots,u_3$ are not distinct. Without loss of generality, suppose $u_0=u_2$. Consider the edges $e_0=v_0v_1$ and $e_1=v_1v_2$. We observe that $e_1\in \ccl[\phi]{e_0}$ because $\phi(e_0)=\phi(e_1)$ and the two edges $e_0$ and $e_1$ are adjacent. Thus, $e_0\in A$ if and only if $e_1\in A$. Now, let $u$ be any vertex of $H$. Because $\phi(e_0)=\phi(e_1)$, the edge $e_0$ has an endpoint in $\phi^{-1}(u)$ if and only if $e_1$ does. It follows that $e_0\in A_u$ if and only if $e_1\in A_u$. Letting $e_2=v_2v_3$ and $e_3=v_3v_0$, an analogous argument shows that $e_2\in A$ if and only if $e_3\in A$, and that $e_2\in A_u$ if and only if $e_3\in A_u$. Thus, the sets
\[
\td E(P)\cap A=\multiset{e_0,e_1,e_2,e_3}\cap A
\quad\text{ and }\quad
\td E(P)\cap A_u=\multiset{e_0,e_1,e_2,e_3}\cap A_u
\]
both have even cardinality. It follows that $\I_A(P)=\I_{A,u}(P)=0$ as desired.
\end{proof}

\begin{corollary}\label{cor:inv}
Fix a graph homomorphism $\phi:G\to H$. Let $\I$ be an intersection invariant associated to $\phi$ of either form $\I_A$ or $\I_{A,u}$. If $P$ and $P'$ are homotopic walks in $G$ then $\I(P)=\I(P')$.
\end{corollary}

\begin{proof}
It suffices to consider the case that $P'$ is formed from $P$ by a single insertion, deletion, or substitution operation from \cref{dfn:homotopy}.

\textit{Insertion}. Suppose $P=v_0\cdots v_k$ and $P'=v_0\cdots v_iwv_i\cdots v_k$. Then $\td E(P')=\td E(P) + \td E(v_iwv_i)$, so
\[
\I(P')=\I(\td E(P) + \td E(v_iwv_i)) =\I(P) + \I(v_iwv_i) = \I(P)
\]
by \cref{prop:hom-inv-helper}(i) and (ii).

\textit{Deletion.} If $P'$ is formed from $P$ by deletion, then $P$ is formed from $P'$ by insertion. Thus, the previous argument shows $\I(P)=\I(P')$.

\textit{Substitution.} Suppose $P=v_0\cdots v_i\cdots v_k$ and $P'=v_0\cdots v'_i\cdots v_k$. Then,
\[\td E(P) + \td E(v_{i-1}v_iv_{i+1}v'_iv_{i-1}) = \td E(P') + \td E(v_{i-1}v_iv_{i+1}v_iv_{i-1}),\]
so \cref{prop:hom-inv-helper}(i) and (iii) imply that
\[\I(P)=\I(P) + \I(v_{i-1}v_iv_{i+1}v'_iv_{i-1}) = \I(P') + \I(v_{i-1}v_iv_{i+1}v_iv_{i-1})=\I(P').\qedhere\]
\end{proof}

We end this section with a result concerning graphs having few distinct homotopy equivalence classes of walks. We call such graphs \emph{simply connected}, in analogy to the topological setting --- a topological space is simply connected if any two continuous paths between the same endpoints are homotopic. (This nomenclature is also suitable because, as shown in \cref{s:topology}, a non-bipartite graph $G$ is simply connected if and only if its neighborhood complex $\NN(G)$ is a simply connected topological space.)

\begin{dfn}\label{dfn:simp-conn}
	Let $G$ be a connected graph. We say $G$ is \emph{simply connected} if it satisfies the converse of \cref{prop:homt-trivial} --- that is, if any two walks of $G$ with the same first and last endpoints and same parity are always homotopy equivalent.
\end{dfn}

\begin{prop}
Suppose $\phi:G\to H$ is a graph homomorphism with $G$ simply connected. Let $A\subseteq E(G)$ be $\phi$-stable and let $\I$ be an $A$-intersection invariant.
\begin{thmenum}
\item If $C$ and $C'$ are two odd closed walks of $G$ then $\I(C)=\I(C')$.\label[prop]{cor:homt-loops}
\item Suppose $\I(C)=1$ for some odd cycle $C$ of $G$. Then the graph $G\setminus A$ is bipartite.\label[prop]{cor:bipartiteB}
\end{thmenum}
\end{prop}

\begin{proof}
(i) Write $C=v_0\cdots v_{k-1}v_0$ and $C'=w_0\cdots w_{\ell-1}w_0$ with $k$ and $\ell$ odd. Because $G$ is connected, there is a walk $P=v_0x_1\cdots x_{m-1}w_0$ from $v_0$ to $w_0$. Let $C''$ be the concatenation
\[
C''=PC'P^{-1}=v_0x_1\cdots x_{m-1}w_0w_1\cdots w_{\ell-1}w_0x_{m-1}\cdots x_1v_0.
\]
Then $C$ and $C''$ are both walks from $v_0$ to $v_0$ of odd lengths $k$ and $2m+\ell$, respectively. By \cref{cor:inv}, we have $\I(C)=\I(C'')$. Moreover, we may write the multiset $\td E(C'')$ as $\td E(C'')=\td E(P)+\td E(P)+\td E(C')$, so
\[
\I(C'')=\I(P)+\I(P)+\I(C')=\I(C')\mod 2.
\]
It follows that $\I(C)=\I(C')$ as desired.

(ii) It suffices to show that every odd cycle $C'$ of $G$ contains an edge of $A$. By (i), we have $\I(C')=\I(C)=1$. If $\I=\I_A$ then this implies that $|E(C')\cap A|$ is odd and therefore nonzero. If $\I=\I_{A,u}$ for some $u\in V(H)$ then we similarly have that $|E(C')\cap A_u|$ is nonzero. Thus, in either case we are able to find an edge of $C'$ that is in $A$.
\end{proof}

\section{Bounding Chromatic Number with a $C_4$ Closure}\label{s:mainthm-reduction}

In this section, we present \cref{keylem1}, which is a key ingredient in the proof of \cref{mainthm}. This constitutes the main application of our homotopy theory.

\begin{lemma}\label{keylem1}
Suppose $\phi:G\to H$ is a graph homomorphism with $G$ simply connected and non-bipartite. Set
\[
s=\min\left\{|\phi(V(C))|: \text{$C$ an odd cycle of $G$}\right\}.
\]
Then there is an edge $f\in E(H)$ contained in an odd cycle of length at most $s$ for which we have
\[\chi(G)\leq\max\left(4,\chi\left(\ccl[H]{f}\right)\right).\]
\end{lemma}

\Cref{keylem1} reduces \cref{mainthm} to an analysis of $\ccl[H]f$ in a $C_{2r+1}$-free graph $H$. The condition that $f$ is in a short odd cycle of $H$ crucially restricts $\ccl[H]f$, by forbidding any configuration of $C_4$'s which lengthens the short odd cycle to one of length $2r+1$.

We separate the proof of \cref{keylem1} into two steps. In \cref{lem:edge-on-short-odd}, we find an edge $e$ of $G$ and an $\ccl[\phi]e$-invariant $\I$ which satisfy the hypotheses of \cref{cor:bipartiteB}, showing that $G\setminus\ccl[\phi]e$ is bipartite. In \cref{lem:extend-coloring}, we show that if $A=\ccl[\phi]e$ does not induce a bipartite subgraph of $G$, then $\chi(G)\leq\chi(\phi(A))$.

\begin{lemma}\label{lem:edge-on-short-odd}
Fix a graph homomorphism $\phi:G\to H$. Suppose $\td F$ is any multiset of edges of $E(G)$ satisfying $\I_{E(G)}(\td F)=1$ and $\I_{E(G),u}(\td F)=0$ for all $u\in V(H)$. Then there is an edge $e\in \td F$ such that $\phi(e)$ is in an odd cycle all of whose edges are in $\phi(\td F)$ and moreover $\I(\td F)=1$ for some $\ccl[\phi]e$-intersection invariant $\I$.
\end{lemma}

\begin{proof}
We proceed by strong induction on $|\td F|$. The statement is vacuously true in the base case $|\td F|=0$, because $|\td F|\equiv \I_{E(G)}(\td F)=1\pmod 2$.

When $|\td F|> 0$, consider the multigraph $X$ on $V(H)$ with edge multiset $\phi(\td F)$. The condition $\I_{E(G),u}(\td F)=0$ implies that every vertex $u$ has even multidegree in $X$. Thus, every connected component of $X$ has an Eulerian tour, so $\phi(\td F)$ may be written as a sum of edge multisets of closed walks of $H$. The condition $\I_{E(G)}(\td F)=1$ is equivalent to $|\td F|$ being odd, so one of these closed walks must have odd length. Call this closed walk $P$, and choose $f\in \td E(P)$ such that $f$ is on an odd cycle in $H$ with all edges in $\td E(P)\subseteq \phi(\td F)$. Choose $e\in \td F$ such that $\phi(e)=f$.

Set $A=\ccl[\phi]e$. If $\I_A(\td F)=1$ or $\I_{A,u}(\td F)=1$ for any $u\in V(H)$ then $e$ satisfies the conditions of the claim. Suppose that $\I_A(\td F)=\I_{A,u}(\td F)=0$ for all $u\in V(H)$.


Write $\td F=\td F_1+\td F_2$, where $\td F_1=\td F\cap A$ and $\td F_2=\td F\setminus A$. We claim that $\I(\td F_1)=0$ for any intersection invariant $\I$ associated to $\phi$. First, we handle $B$-intersection invariants $\I$ such that $B=\ccl[\phi]{e'}$ for some $e'\in E(G)$, using the fact that $\I(\td F')=\I(\td F'\cap B)$ for any multiset $\td F'$. Observe that the equivalence class $B=\ccl[\phi]{e'}$ is either equal to or disjoint from $A=\ccl[\phi]e$. If $B=A$ then $\I(\td F_1)=\I(\td F\cap B)=\I(\td F)=0$, and if $B\neq A$ then $\td F_1\subseteq A$ is disjoint from $B$, so $\I(\td F_1)=\I(\td F_1\cap B)=\I(\emptyset)=0$. Now, suppose $\I$ is a $B$-intersection invariant with $B$ an arbitrary $\phi$-stable set. We may partition $B$ as $\mathcal P=\left\{\ccl[\phi]{e'}:e'\in B\right\}$, allowing us to write
\[\I_B(\td F_1)=\sum_{B'\in\mathcal P}\I_{B'}(\td F_1)=0\qquad\text{and}\qquad\I_{B,u}(\td F_1)=\sum_{B'\in\mathcal P}\I_{B',u}(\td F_1)=0.\]
Hence, $\I(\td F_1)=0$ for any intersection invariant $\I$ associated to $\phi$.

Now, consider the multiset $\td F_2$. For any intersection invariant $\I$ associated to $\phi$, we have $\I(\td F_1)=0$, so $\I(\td F_2)=\I(\td F_1)+\I(\td F_2)=\I(\td F)$. In particular, we have $\I_{E(G)}(\td F_2)=\I_{E(G)}(\td F)=1$ and $\I_{E(G),u}(\td F_2)=\I_{E(G),u}(\td F)=0$ for all $u\in V(H)$. Applying the inductive hypothesis to $\td F_2$ yields an edge $e'\in \td F_2$ such that $\phi(e')$ is on an odd cycle all of whose edges are in $\phi(\td F_2)$ and furthermore $\I(\td F_2)=1$ for some $\ccl[\phi]{e'}$-intersection invariant $\I$. Because $\phi(\td F)\supseteq\phi(\td F_2)$ and $\I(\td F)=\I(\td F_2)$, it follows that the edge $e'$ satisfies the desired conditions.
\end{proof}

\begin{lemma}\label{lem:extend-coloring}
Fix a graph homomorphism $\phi:G\to H$ with $G$ simply connected. Suppose $E(G)=A\sqcup B$ is a partition of $E(G)$ into two $\phi$-stable sets and let $G_A$ and $G_B$ be the corresponding edge-induced subgraphs of $G$. Suppose that $G_A$ is connected and non-bipartite. Then $\chi(G)\leq\chi(\phi(G_A))$.
\end{lemma}

\begin{proof}
Our proof will show the following two statements.
\begin{thmenum}
\item Every walk $P$ in $G$ with all edges in $B$ and with endpoints $v,v'\in V(G_A)$ has even length and its endpoints satisfy $\phi(v)=\phi(v')$.
\item Let $\gamma_0:V(G_A)\to[r]$ be any proper $r$-coloring of $G_A$ such that $\gamma_0(v)=\gamma_0(v')$ whenever $\phi(v)=\phi(v')$. Then $\gamma_0$ can be extended to a proper $r$-coloring $\gamma$ of $G$.
\end{thmenum}
Our desired result follows immediately from (ii) by taking $r=\chi(\phi(G_A))$ and letting $\gamma_0$ be the pullback of any $r$-coloring of $\phi(G_A)$.

We begin by showing (i). Because $G_A$ is non-bipartite, there is an odd cycle $C\subseteq G_A$. Because $E(C)\subseteq A$ is disjoint from $B$, we have $\I_B(C)=\I_{B,u}(C)=0$ for any $u\in V(H)$. By \cref{cor:homt-loops}, we have $\I_B(C')=\I_{B,u}(C')=0$ for any odd closed walk $C'$ in $G$.

Let $P=w_0\cdots w_k$ be any walk in $G$ with all edges in $B$ and with $w_0,w_k\in V(G_A)$. Because $G_A$ is connected and contains an odd cycle, there are both odd and even walks between any two vertices of $V(G_A)$. Pick a walk $P'=w_kw_{k+1}\cdots w_{\ell-1}w_0$ such that the concatenation $C'=PP'=w_0\cdots w_{\ell-1}w_0$ has odd length.

We have $0=\I_B(C')\equiv|E(C')\cap B|=k\pmod 2$, so $P$ has even length. Moreover, for any $u\in V(H)$, each internal vertex $w_i$ ($0<i<k$) contributes either 0 or 2 to $|\td E(P)\cap B_u|$, so
\[
\I_{B,u}(C')=\I_{B,u}(P)\equiv\left|\multiset{w_0,w_k}\cap\phi^{-1}(u)\right|\mod 2.
\]
Thus, the condition that $\I_{B,u}(C')=0$ for all $u\in V(H)$ implies that $\phi(w_0)=\phi(w_k)$. This completes the proof of (i).

We now use (i) to show (ii). Extend $\gamma_0$ to a coloring $\gamma:V(G)\to[r]$ as follows. For each vertex $v\in V(G)$, choose any walk $P=vv_1\cdots v_k$ with all edges in $B$ and with $v_k\in V(G_A)$. Set
\[
\gamma(v)=\left\{\begin{array}{ll}\gamma_0(v_k)&\text{if $k$ is even,}\\\text{anything in }[r]\setminus\gamma_0(v_k)&\text{if $k$ is odd.}\end{array}\right.
\]
We claim that the choice of $\gamma(v)$ is independent of the choice of $P$. Let $P'=vw_1\cdots w_\ell$ be another walk from $v$ with all edges in $B$ and with $w_\ell\in V(G_A)$. Then the concatenation
\[
P^{-1}P'=v_k\cdots v_1vw_1\cdots w_\ell
\]
satisfies the hypotheses of (i). In particular, (i) implies that $\phi(v_k)=\phi(w_\ell)$ and that the length $k+\ell$ of $P^{-1}P'$ is even, so $\gamma_0(v_k)=\gamma_0(w_\ell)$ and $k\equiv\ell\pmod 2$. It follows that $\gamma(v)$ does not depend on the path $P$ chosen, and thus the coloring $\gamma$ is well-defined.

We now show that $\gamma$ is a proper coloring of $G$. For vertices $v\in V(G_A)$, we have that $\gamma(v)=\gamma_0(v)$, by choosing $P$ to be the length-0 walk $v$. Because $\gamma_0$ is a proper coloring of $G_A$, it follows that $\gamma(v)\neq\gamma(v')$ for any edge $vv'\in A$. Now, consider an edge $vv'\in B$. Pick a walk $P=vv_1\cdots v_k$ with each edge in $B$ and with $v_k\in V(G_A)$, and let $P'=(v'v)P=v'vv_1\cdots v_k$. Because the walks $P$ and $P'$ have lengths of opposite parities, one of $\gamma(v)$ and $\gamma(v')$ is equal to $\gamma_0(v_k)$ and the other is in $[r]\setminus\gamma_0(v_k)$, implying that $\gamma(v)\neq\gamma(v')$. Thus, $\gamma$ is a proper $r$-coloring of $G$ that extends $\gamma_0$, completing the proof of (ii).
\end{proof}

\begin{proof}[Proof of \cref{keylem1}.]
Let $C$ be an odd cycle of $G$ with $|\phi(V(C))|=s$. By \cref{lem:edge-on-short-odd} applied to the edge (multi)set $E(C)$, there is an edge $e$ of $C$ with $\phi(e)$ on an odd cycle in $\phi(C)$, which odd cycle has at most $|\phi(V(C))|=s$ vertices. Moreover, setting $A=\ccl[\phi]e$, the lemma yields an $A$-intersection invariant $\I$ with $\I(C)=1$. We take $f=\phi(e)$.

Let $B=E(G)\setminus A$. By \cref{cor:bipartiteB}, $B$ induces a bipartite subgraph of $G$. If $A$ also induces a bipartite subgraph of $G$ then we have
\[
\chi(G)\leq\chi(A)\times\chi(B)\leq 4.
\]
If the subgraph of $G$ induced by $A$ is not bipartite, then \cref{lem:extend-coloring} implies
\[
\chi(G)\leq\chi(\phi(A))\leq\chi\left(\ccl[H]f\right),
\]
because $A=\ccl[\phi]e\subseteq\phi^{-1}\left(\ccl[H]f\right)$.
\end{proof}

\section{Proof of Main Theorem}\label{s:mainthm-pf}

We may now complete the proof of \cref{mainthm}. As mentioned at the beginning of \cref{s:mainthm-reduction}, the proof can be reduced to an analysis of $\ccl[H]f$, where $f$ is an edge on a short odd cycle in the $C_{2r+1}$-free graph $H$.

The proof of \cref{mainthm} relies on \cref{keylem2-weak,keylem3}, two intermediate results concerning the structure of $C_{2r+1}$-free graphs. We prove \cref{mainthm} using these lemmas, and subsequently prove \cref{keylem2-weak} in \cref{s:keylem2} and \cref{keylem3} in \cref{s:keylem3}.

\begin{lemma}\label{keylem2-weak}
Fix $r\geq 2$. Suppose $H$ is a $C_{2r+1}$-free graph. Let $C$ be a shortest odd cycle of $H$ and let $f$ be an edge of $C$. If $|C|<2r+1$ then every edge in $\ccl[H]{f}$ is on an odd cycle of length less than $2r+1$ containing at least one vertex of $C$.
\end{lemma}

\begin{lemma}\label{keylem3}
Fix $r\geq 1$. Let $H$ be a $C_{2r+1}$-free graph and $v$ a vertex of $H$ such that every vertex of $H$ is within distance $r$ of $v$. Then $\chi(H)\leq 4r$.
\end{lemma}

We now prove \cref{mainthm}, which is restated here for the reader's convenience.

\begin{reptheorem}{mainthm}
Let $G$ be a simply connected graph and fix $r\geq 2$. Suppose $\phi:G\to H$ is a graph homomorphism such that $H$ is $C_{2r+1}$-free, and suppose further that some odd cycle $C$ of $G$ satisfies $|\phi(V(C))|\leq 2r+2$. Then $\chi(G)<8r^2$.
\end{reptheorem}

\begin{proof}
By hypothesis, $G$ contains an odd cycle and is therefore not bipartite. By \cref{keylem1}, there is an edge $f\in E(H)$ contained in an odd cycle of length at most $2r+2$ such that $\chi(G)\leq\max\left(4,\chi\left(\ccl[H]f\right)\right)$. Because $H$ is $C_{2r+1}$-free, this odd cycle has length at most $2r-1$.

Let $C$ be a shortest cycle among those odd cycles of $H$ which contain an edge of $\ccl[H]f$. By the above reasoning, we have $|C|\leq 2r-1$. Let $f'$ be an edge in $E(C)\cap\ccl[H]f$. Let $H_1$ be the subgraph of $H$ induced by the edge set $E(C)\cup\ccl[H]f$.

Observe that $H_1\setminus\ccl[H]f$ is acyclic, as its edge set is a subset of $E(C)\setminus\{f'\}$. Hence, every cycle of $H_1$ contains an edge of $\ccl[H]f$. This implies that $C$ is a shortest odd cycle of $H_1$. Applying \cref{keylem2-weak} in $H_1$, we see that every edge of $\ccl[H_1]{f'}$ is on an odd cycle of length at most $2r-1$ containing at least one vertex of $C$. Furthermore, because every $C_4$ in $\ccl[H]f=\ccl[H]{f'}$ is also contained in $H_1$, we have $\ccl[H]f=\ccl[H]{f'}=\ccl[H_1]{f'}$. It follows that every vertex of $V\left(\ccl[H]f\right)$ is within distance $r-1$ in $H_1$ of some vertex of $C$.

For each vertex $v\in V(C)$, let $N_{\leq r}(v)\subseteq V(H_1)$ be the set of vertices within distance $r$ of $v$ in $H_1$. By \cref{keylem3}, the induced subgraph $H_1[N_{\leq r}(v)]$ is $(4r)$-colorable. Using a disjoint set of $4r$ colors for each set $N_{\leq r}(v)$, we may color $H_1$ with $|C|\times 4r<8r^2$ colors, showing $\chi(\ccl[H]f)\leq\chi(H_1)<8r^2$. Hence, $\chi(G)\leq\max(4,\chi(\ccl[H]f))< 8r^2$.
\end{proof}

\begin{rmk}\label{rmk:mainthm-tight}
	With more care in the above proof, one could improve the upper bound on $\chi(G)$ by a constant factor. We forgo these small improvements for simplicity, as any finite upper bound on $\chi(G)$ suffices to prove \cref{thm:dhom>0}. 
\end{rmk}

\subsection{Proof of \cref{keylem2-weak}}\label{s:keylem2}

Let $H$ be a $C_{2r+1}$-free graph and $C$ a shortest odd cycle in $H$. Suppose $|C|<2r+1$ and fix $f\in E(C)$. Observe that any $e\in\ccl[H]f$ can be obtained from some edge $e_0\in E(C)$ via a sequence of the form
\[
e_0\in D_1\ni e_1\in D_2\ni\cdots\in D_m\ni e_m=e,
\]
where each $D_i$ is a $C_4$ of $H$ and each $e_i$ is an edge contained in both $D_i$ and $D_{i+1}$. We will find it useful to assume that $e_1,\ldots,e_m\notin E(C)$, which may be assured by choosing a shortest such sequence.

Our proof iteratively constructs short odd cycles $X_0,X_1,\ldots,X_m$ with $e_i\in E(X_i)$. We show inductively that $|X_i|<2r+1$, using the following preorder on odd cycles of $H$.
\begin{dfn}\label{dfn:cycle-cons}
Let $X$ and $Y$ be odd cycles in a graph $H$. We write $Y\prec X$ if $H$ contains cycles of all odd lengths $\ell$ with $|X|\leq\ell\leq|Y|$.
\end{dfn}

\begin{rmk}
The statement $Y\prec X$ should be interpreted as saying that $Y$ cannot be much longer than $X$. We note that, although $|Y|\leq|X|$ implies $Y\prec X$, the converse does not necessarily hold. Indeed, we often show $Y\prec X$ by proving that $|Y|\leq|X|+2$.
\end{rmk}

Notice that if $X$ is an odd cycle in $H$ with $|X|<2r+1$ then, because $H$ is $C_{2r+1}$-free, any $Y\prec X$ will also satisfy $|Y|<2r+1$. Hence, if our odd cycles $X_i$ satisfy the relations
\begin{equation}\label{eqn:prec}
	X_m\prec X_{m-1}\prec\cdots\prec X_1\prec X_0=C,
\end{equation}
then the condition $|C|<2r+1$ implies that $|X_m|<2r+1$.

We construct the odd cycles $X_i$ using a type of auxiliary path, called an ear, that is often relevant to questions in graph connectivity. A \emph{$C$-ear} is a path $Q$ of $H$ whose endpoints are (distinct) vertices $v,v'\in V(C)$ and whose interior vertices do not lie on $C$. We additionally require that $C$-ears have length at least 2, so that $E(Q)$ is disjoint from $E(C)$. The endpoints of any $C$-ear $Q$ divide $C$ into two paths of opposite parities. Concatenating $Q$ with either subpath of $C$ forms two cycles of opposite parities; let $\Co(Q)$ be the cycle of odd length among them. Our inductive step (\cref{prop:e1-to-e2-weak}) shows that if $e_i$ and $e_{i+1}$ are edges of the same $C_4$ and $Q_i$ is a $C$-ear containing $e_i$, then there is a $C$-ear $Q_{i+1}$ containing $e_{i+1}$ with $\Co(Q_{i+1})\prec\Co(Q_i)$. This allows us to inductively construct a sequence $Q_1,\ldots,Q_m$ of $C$-ears such that the odd cycles $X_i=\Co(Q_i)$ satisfy (\ref{eqn:prec}).

We begin the proof of \cref{keylem2-weak} with a proposition that allows us to bound the difference between two odd cycles.

\begin{prop}\label{prop:symdiffweak}
Let $H$ be a graph and let $C$ be a shortest odd cycle in $H$. Let $X$ and $Y$ be odd cycles in $H$ and set $X_1=E(X)\setminus E(C)$ and $Y_1=E(Y)\setminus E(C)$. Then $|Y|\leq|X|+2|Y_1\setminus X_1|$.
\end{prop}

We find it convenient to prove \cref{prop:symdiffweak} in the following generalized form. Say a graph is \emph{Eulerian} if every vertex has even degree; we note that an Eulerian graph need not be connected. Similarly, say that a set of edges $E\subseteq E(H)$ is \emph{Eulerian} if $E$ induces an Eulerian subgraph of $H$.

\begin{prop}\label{prop:symdiff}
Let $H$ be a graph and let $C$ be a shortest odd cycle in $H$. Let $X,Y\subseteq E(H)$ be Eulerian edge sets with $|X|\equiv|Y|\pmod 2$. Partition $X=X_1\sqcup X_2$, where $X_1=X\setminus E(C)$ and $X_2=X\cap E(C)$, and partition $Y=Y_1\sqcup Y_2$ analogously. Then
\[
|Y|\leq |X|+2|Y_1\setminus X_1|.
\]
\end{prop}

\begin{proof}
We first observe that for any sets $S$ and $T$, we have
\[
|S|-|T|+|S\sd T|=|S| -|T| + |S\cup T| - |S\cap T|=\left(|S|-|S\cap T|\right) + \left(|S\cup T|-|T|\right) = 2|S\setminus T|.
\]
Thus, we have
\begin{align*}
|Y|-|X|&=|Y_1|+|Y_2|-|X_1|-|X_2|\\
&=\left(|Y_1|-|X_1|\right)-\left(|X_2|-|Y_2|\right)
\\&=\left(|Y_1|-|X_1|+|X_1\sd Y_1|\right) - \left(|X_2|-|Y_2|+|X_2\sd Y_2|\right) - \left(|X_1\sd Y_1| - |X_2\sd Y_2|\right)
\\&=2|Y_1\setminus X_1|-2|X_2\setminus Y_2|-\left(|X_1\sd Y_1| - |X_2\sd Y_2|\right).
\end{align*}

We now consider the last term. We have
\[|X_1\sd Y_1|-|X_2\sd Y_2|
=|X_1\sd Y_1| + |X_2\sd Y_2\sd E(C)| - |C|
=|X\sd Y\sd E(C)|-|C|.\]
Let $Z=X\sd Y\sd E(C)$. The symmetric difference of Eulerian sets is Eulerian, so $Z$ is Eulerian. Additionally, $|Z|$ is odd because $|X|\equiv|Y|\pmod 2$. It follows that $|Z|\geq|C|$, because every odd-cardinality Eulerian set contains the edge set of an odd cycle, and $C$ is a shortest odd cycle in $H$. Thus, $|X_1\sd Y_1|-|X_2\sd Y_2|=|Z|-|C|\geq 0$. It follows that
$|Y|-|X|\leq2|Y_1\setminus X_1|-2|Y_2\setminus X_2|$,
as desired.
\end{proof}

The next two propositions form the core of our inductive step.

\begin{prop}\label{prop:edgeinA-weak}
Let $C$ be a shortest odd cycle in a graph $H$. Let $e_1=xy_1$ be an edge of $H$ on a $C$-ear $Q_1$ and let $e_2=xy_2$ be an adjacent edge satisfying $y_2\in V(Q_1)\cup V(C)$. If $e_2\notin E(C)$ then $e_2$ is on a $C$-ear $Q_2$ with $\Co(Q_2)\prec\Co(Q_1)$. Moreover, $V(Q_2)\subseteq V(Q_1)\cup V(C)$, and either $y_1xy_2$ is a subpath of $Q_2$ or $y_1\notin V(Q_2)\cup V(C)$.
\end{prop}

\begin{proof}
Write $Q_1=w_0\cdots w_m$, with the orientation chosen so that $w_m$ is neither $x$ nor $y_1$. Write $x=w_i$ and set
\[Q_2=\left\{\begin{array}{ll}
w_0\cdots w_jw_i\cdots w_m&\,\,\text{if $y_2=w_j$ for some $j<i$,}
\\w_0\cdots w_iw_j\cdots w_m&\,\,\text{if $y_2=w_j$ for some $j>i$,}
\\w_0\cdots w_iy_2&\,\,\text{if $y_2\in V(C)\setminus \{w_0,w_m\}$.}
\end{array}\right.\]
It is clear in all cases that $Q_2$ is a $C$-ear containing $e_2$ and that $V(Q_2)\subseteq V(Q_1)\cup\{y_2\}\subseteq V(Q_1)\cup V(C)$. If $y_1=w_{i\pm 1}$ is on $Q_2$ then it is immediate that $y_1xy_2=w_{i\pm 1}w_iy_2$ is a subpath of $Q_2$. If $y_1$ is on $C$, then we have $y_1=w_0$ and $i=1$, so $y_1xy_2=w_0w_1y_2$ will be a subpath of $Q_2$. Thus, if $y_1\in V(Q_2)\cup V(C)$ then $y_1xy_2$ is a subpath of $Q_2$. It only remains to show that $\Co(Q_2)\prec\Co(Q_1)$.

Take $X=\Co(Q_1)$ and $Y=\Co(Q_2)$. To show that $Y\prec X$, it suffices to show that $|Y|\leq|X|+2$. Observe that $E(X)\setminus E(C)=E(Q_1)$ and $E(Y)\setminus E(C)=E(Q_2)$. Then, \cref{prop:symdiffweak} yields that
\[
|Y|\leq |X|+2|E(Q_2)\setminus E(Q_1)|=|X|+2|\{e_2\}|=|X|+2
\]
in each of the three possible constructions for $Q_2$, as desired.
\end{proof}

\begin{prop}\label{prop:e1-to-e2-weak}
Let $C$ be a shortest odd cycle in a graph $H$. Let $D$ be a $C_4$ edge-disjoint from $C$ and let $e_1$ and $e_2$ be distinct edges of $D$. Given any $C$-ear $Q_1$ containing $e_1$, there is a $C$-ear $Q_2$ containing $e_2$ with $\Co(Q_2)\prec\Co(Q_1)$.
\end{prop}

\begin{proof}
We first handle the case that $e_1$ and $e_2$ are adjacent edges of $D$. Write $D=x_1x_2x_3x_4$ with $e_1=x_1x_2$ and $e_2=x_2x_3$. We use casework on whether $x_3$ and $x_4$ are in the set $W=V(Q_1)\cup V(C)$.

If $x_3\in W$ then we may apply \cref{prop:edgeinA-weak} to construct $Q_2$.
If $x_3,x_4\notin W$, then construct $Q_2$ by replacing the subpath $x_1x_2$ of $Q_1$ with $x_1x_4x_3x_2$. Because $Q_1$ and $Q_2$ have the same endpoints and lengths of the same parity, it follows that $\Co(Q_1)\cap C=\Co(Q_2)\cap C$, and thus that $|\Co(Q_2)|=|\Co(Q_1)|+2$. Hence, $\Co(Q_2)\prec\Co(Q_1)$.

Lastly, suppose $x_3\notin W$ but $x_4\in W$. Write $e_4=x_4x_1$. By \cref{prop:edgeinA-weak}, there is a $C$-ear $Q_4$ with $\Co(Q_4)\prec\Co(Q_1)$. Moreover, $V(Q_4)\subseteq V(Q_1)\cup V(C)$, and either $Q_4$ contains $x_4x_1x_2$ as a subpath or $x_2$ is not in the set $W'=V(Q_4)\cup V(C)$. Notice that $x_3\notin W'$, as $W'\subseteq W$.
If $x_4x_1x_2$ is a subpath of $Q_4$ then we construct $Q_2$ by substituting $x_3$ for $x_1$ in $Q_4$. This yields a valid $C$-ear because $x_3\notin W'$. Furthermore, $|\Co(Q_2)|=|\Co(Q_4)|$, so $\Co(Q_2)\prec\Co(Q_1)$. If instead $x_2\notin W'$, then we construct $Q_2$ by replacing the subpath $x_1x_4$ of $Q_4$ with $x_1x_2x_3x_4$. This yields a valid $C$-ear because $x_2,x_3\notin W'$. Additionally, $|\Co(Q_2)|=|\Co(Q_4)|+2$ by the same reasoning as above, so $\Co(Q_2)\prec\Co(Q_4)\prec\Co(Q_1)$. Thus, if $e_1$ and $e_2$ are adjacent edges of $D$, then there is a $C$-ear $Q_2$ as desired.

Now, suppose $e_1$ and $e_2$ are opposite edges of $D$, and let $e'\in E(D)$ be adjacent to them both. Repeating the above argument twice, we locate a $C$-ear $Q'$ containing $e'$ with $\Co(Q')\prec\Co(Q_1)$, and a $C$-ear $Q_2$ containing $e_2$ with $\Co(Q_2)\prec\Co(Q')$. Then, $\Co(Q_2)\prec\Co(Q')\prec\Co(Q_1)$, as desired.
\end{proof}

We may now use \cref{prop:e1-to-e2-weak} to prove \cref{keylem2-weak}.

\begin{proof}[Proof of \cref{keylem2-weak}]
As discussed at the beginning of this section, every edge $e\in\ccl[H]{f}$ may be obtained from some edge $e_0\in E(C)$ via a sequence of the form
\[
e_0\in D_1\ni e_1\in D_2\ni\cdots\in D_m\ni e_m=e,
\]
where each $D_i$ is a $C_4$ in $H$ and each $e_i$ is an edge contained in both $D_i$ and $D_{i+1}$. We choose a shortest such sequence, guaranteeing that the cycles $D_2,\ldots,D_m$ are each edge-disjoint from $C$ and that $e_i\notin E(C)$ for $i\geq 1$.

We show by induction on $i\geq 1$ that every edge $e_i$ is on a $C$-ear $Q_i$ with $\Co(Q_i)\prec C$. For the base case $i=1$, observe that $e_1$ has at least one endpoint not in $V(C)$. Indeed, if both endpoints of $e_1$ were on $C$ then either $e_1\in E(C)$ (contradicting the minimality of $m$) or $e_1$ would be a chord of $C$ (creating an odd cycle shorter than $C$). It follows that there is a $C$-ear $Q_1$ of length 2 or 3 with $e_1\in E(Q_1)\subseteq E(D_1)\setminus\{e_0\}$.  We have that
\[
|\Co(Q_1)|\leq|E(Q_1)| + (|C|-1)\leq |C|+2,
\]
so $\Co(Q_1)\prec C$.

The inductive step follows directly from \cref{prop:e1-to-e2-weak}. Let $i>1$, and suppose that $e_{i-1}$ is on a $C$-ear $Q_{i-1}$ with $\Co(Q_{i-1})\prec C$. Applying \cref{prop:e1-to-e2-weak} to the pair $(e_{i-1},e_i)$ yields a $C$-ear $Q_i$ with $\Co(Q_i)\prec\Co(Q_{i-1})\prec C$. Thus, each edge $e_i$ is on a $C$-ear $Q_i$ with $\Co(Q_i)\prec C$. 

We claim that the cycle $\Co(Q_m)$ satisfies the desired conditions. Certainly $\Co(Q_m)$ contains $e=e_m$ as well as vertices of $C$. Moreover, because $|C|<2r+1$ and $H$ is $C_{2r+1}$-free, it follows that $|\Co(Q_m)|<2r+1$.
\end{proof}

\subsection{Proof of \cref{keylem3}}\label{s:keylem3}

Given a graph $H$ and a vertex $v$ of $H$, let $N_i(v)$ denote the set of vertices of $H$ at distance $i$ from $v$. We show that if $H$ is a $C_{2r+1}$-free graph and $v\in V(H)$ a fixed vertex, then each of the induced subgraphs $H[N_0(v)],H[N_1(v)],\ldots,H[N_r(v)]$ is $2r$-colorable. 

Our strategy is motivated by Thomassen \cite{Th07}'s proof that $\dhom(C_{2r+1})=0$ for any $r\geq 2$. A key step in the proof is to bound the chromatic number of any $C_{2r+1}$-free graph $H$ having radius at most 2 from a center $v$ --- thereby deriving a weaker version of \cref{keylem3} --- by separately considering the chromatic numbers of $H[N_1(v)]$ and $H[N_2(v)]$. We remark that, if $G$ is a $C_{2r+1}$-free graph with minimum degree $\del(G)\geq\al|G|$, then combining \cref{keylem3} with Thomassen's approach can derive better explicit bounds (in terms of $\al$) on $\chi(G)$.

To bound the chromatic number of each neighborhood, we use the following stronger property.
\begin{dfn}
	A graph $H'$ is \emph{$k$-degenerate} if every subgraph of $H'$ has a vertex of degree at most $k$.
\end{dfn}
It is well-known that any $k$-degenerate graph is $(k+1)$ colorable, so it suffices to show that each subgraph $H[N_i(v)]$ for $0\leq i\leq r$ is $(2r-1)$-degenerate.

\begin{lemma}\label{lem:degenerate-N}
Fix $r\geq 1$. Let $H$ be a $C_{2r+1}$-free graph and $v$ a vertex of $H$. Suppose $J$ is a connected nonempty subgraph of $H$ whose vertex set $W=V(J)$ satisfies $W\subseteq N_s(v)$ for some $s\leq r$. Then $J$ has minimum degree less than $2r$.
\end{lemma}

\begin{proof}
We proceed by contradiction. Suppose $J$ has minimum degree at least $2r$.

Fix a function $\alpha:V(H)\setminus\{v\}\to V(H)$ such that for each $i\geq 1$ and each $w\in N_i(v)$, the image $\alpha(w)$ is a ``parent'' of $w$, i.e.\ a neighbor of $w$ in $N_{i-1}(v)$. Observe that the function $\al^s:N_s(v)\to N_0(v)$ is constant on $W$, as each $w\in W$ is mapped to $v$. We show by induction on $t=s,s-1,\ldots,0$ that the function $\al^t:N_s(v)\to N_{s-t}(v)$ is constant on $W$. We then derive a contradiction by observing that the identity function $\al^0:N_s(v)\to N_s(v)$ is not constant on $W$.

Fix any $t$ with $0<t\leq s$ and suppose that $\al^t:N_s(v)\to N_{s-t}(v)$ is constant on $W$. Given $w,w'\in W$, we show $\al^{t-1}(w)=\al^{t-1}(w')$. We use casework on the length of a path between $w$ and $w'$ in $J$, first assuming that they are connected by a path of length exactly $2r-2t+1$ or 2, and then handling the general case.

Suppose $w$ and $w'$ are connected by a path $P$ of length $2r-2t+1$ in $J$. Consider the closed walk
\[
Q=\al^t(w)\al^{t-1}(w)\cdots\al(w)wPw'\al(w')\cdots\al^t(w')
\]
of length $2r+1$ with endpoints $\alpha^t(w)=\alpha^t(w')$. Because $H$ is $C_{2r+1}$-free, $Q$ is not a cycle, and hence visits some vertex twice. However, if $\al^{t-1}(w)$ and $\al^{t-1}(w')$ were distinct vertices, then we would have $\alpha^i(w)\neq \alpha^i(w')$ for each $0\leq i<t$, in which case $Q$ would repeat no vertex. It follows that $\al^{t-1}(w)=\al^{t-1}(w')$ for any two vertices $w,w'\in W$ connected by a path of length $2r-2t+1$ in $J$.

Now, suppose $w$ and $w'$ are connected by a path $wu_1w'$ of length 2 in $J$. Construct a path $u_1u_2\cdots u_{2r-2t+1}$ in $J$ as follows. For each $2\leq i\leq 2r-2t+1$, choose $u_i$ to be a neighbor of $u_{i-1}$ in $J$ distinct from the $i$ vertices $w,w',u_1,\ldots,u_{i-2}$. Such a vertex $u_i$ exists because there are at most $2r-2t+1\leq 2r-1$ forbidden vertices, but $u_{i-1}$ has degree at least $2r$ in $J$. The two vertices $w$ and $u_{2r-2t+1}$ are connected by a path of length $2r-2t+1$ in $J$, as are the two vertices $w'$ and $u_{2r-2t+1}$, so it follows that $\al^{t-1}(w)=\al^{t-1}(u_{2r-2t+1})=\al^{t-1}(w')$. Thus, $\al^{t-1}(w)=\al^{t-1}(w')$ for any two vertices $w,w'\in W$ connected by a path of length $2$ in $J$.

Now, consider a path $P=w_0\cdots w_k$ of arbitrary length in $J$. We will show that $\al^{t-1}(w_0)=\al^{t-1}(w_k)$. If $k$ is even, then
\[\al^{t-1}(w_0)=\al^{t-1}(w_2)=\cdots=\al^{t-1}(w_k),\]
 by writing $P$ as a concatenation of paths of length 2. If $k$ is odd and $k\geq 2r-2t+1$ then we have
\[\al^{t-1}(w_0)=\al^{t-1}(w_{2r-2t+1})=\al^{t-1}(w_{2r-2t+3})=\cdots=\al^{t-1}(w_k),\]
by writing $P$ as a concatenation of a path of length $2r-2t+1$ followed by some paths of length 2. Lastly, suppose $k$ is odd and $k<2r-2t+1$. Extend $P$ to a path $w_0\cdots w_kw_{k+1}\cdots w_{2r-2t+1}$ in $J$ as follows. For each $k<i\leq 2r-2t+1$, choose $w_i$ to be a neighbor of $w_{i-1}$ in $J$ distinct from the $i-1$ vertices $w_0,w_1,\ldots,w_{i-2}$. We then have
\[
\al^{t-1}(w_k)=\al^{t-1}(w_{k+2})=\cdots=\al^{t-1}(w_{2r-2t+1})=\al^{t-1}(w_0).
\]
Thus we have $\al^{t-1}(w_0)=\al^{t-1}(w_k)$, no matter the value of $k$. This completes the proof that if $\al^t$ is constant on $W$ then $\al^{t-1}$ is also constant on $W$.

Because $\al^s$ is constant on $W$, it follows by induction that $\al^{s-1},\al^{s-2},\ldots,\al^0$ are all constant on $W$. However, the identity function $\al^0:N_s(v)\to N_s(v)$ is not constant on $W$, because the minimum degree condition implies that $W$ contains at least $2r+1>1$ distinct vertices. This is the desired contradiction.
\end{proof}

\begin{proof}[Proof of \cref{keylem3}]
\Cref{lem:degenerate-N} implies that for each $s\leq r$, the induced subgraph $H[N_s(v)]$ is $(2r-1)$-degenerate and thus $2r$-colorable. Additionally, observe that if $|s-t|\geq 2$ then the sets $N_s(v)$ and $N_t(v)$ are disconnected. Thus, we may $(4r)$-color $H$ by using $2r$ colors to color the set
\[N_0(v)\cup N_2(v)\cup\cdots\cup N_{2\lfloor r/2\rfloor}\]
of vertices at even distance from $v$, and using a disjoint set of $2r$ colors for the set of vertices at odd distance from $v$.
\end{proof}

\section{The Borsuk Graph}\label{s:borsuk}

Having proven \cref{mainthm}, we turn our attention to the homomorphism threshold of odd cycles. To prove \cref{thm:dhom>0}, we consider a sequence of approximations to a Borsuk graph, and use \cref{mainthm} to show that this family has no finite $C_{2r+1}$-free homomorphic image.

Let $S^n$ denote the $n$-dimensional unit sphere. We equip $S^n$ with the metric $d$ measuring distance along the sphere; for example, any point $x\in S^n$ and its antipode $-x$ have distance $d(x,-x)=\pi$. We write $B_\eps(x)\subseteq S^n$ for the open ball of radius $\eps$ around any point $x\in S^n$. Denote the normalized surface measure of this ball by $\mu(n,\eps)$; that is, $\mu(n,\eps)$ is the fraction of the surface of $S^n$ covered by a spherical cap of radius $\eps$.

As defined in \cref{s:borsukintro}, the Borsuk graph $\BG(n,\eps)$ is an infinite graph on $S^n$, with an edge between two points $x,y$ if $d(x,-y)<\eps$. It is well known that $\BG(n,\eps)$ has chromatic number at least $n+2$, with equality if $\eps$ is sufficiently small, and that this fact is equivalent to the Borsuk--Ulam theorem (see \cite{ErHa67,Lo83}).
We further note that, if $\eps\leq \pi/(2r+1)$ for some integer $r$, then $\BG(n,\eps)$ is $C_{2r+1}$-free. Indeed, the $(2r+1)^\text{th}$ neighborhood of any vertex $x$ is $B_{(2r+1)\eps}(-x)$, which only contains $x$ when $(2r+1)\eps>\pi$.

Consider a family of increasingly good approximations to the $C_{2r+1}$-free graph $\BG\left(n,\frac{\pi}{2r+1}\right)$. The goal of this section is to apply \cref{mainthm} to show that this family has no finite $C_{2r+1}$-free homomorphic image. Approximations to Borsuk graphs are formally defined as follows.

\begin{dfn}\label{dfn:BG-approx}
Fix $n\geq 1$ and $0<\del<\eps$. Let $V\subseteq S^n$ be a finite set which is antipodally closed, i.e.\ for each $v\in V$, its antipode $-v$ is also in $V$. Let $G$ be the subgraph of $\BG(n,\eps)$ induced by $V$. If for each $x\in S^n$, there is some $v\in V$ with $d(v,x)<\del$, then we say $G$ is a \emph{$\del$-approximation} to $\BG(n,\eps)$.
\end{dfn}

Applying \cref{mainthm} requires three intermediate results. We show that if $\del$ is sufficiently small, then any $\del$-approximation $G$ to $\BG\left(n,\frac{\pi}{2r+1}\right)$ has large chromatic number (\cref{prop:approx-chi}) and is simply connected (\cref{prop:approx-simp-conn}). Additionally, if $\phi:G\to H$ is a graph homomorphism with $|H|$ sufficiently small in terms of $\del$, we show that $\phi$ acts noninjectively on some $C_{2r+3}$ of $G$ (\cref{prop:approx-noninj}).

\begin{prop}\label{prop:approx-chi}
Fix $n\geq 1$ and $0<\del<\eps/2$. If $G$ is a $\del$-approximation to $\BG(n,\eps)$ then $\chi(G)\geq n+2$.
\end{prop}

\begin{proof}
We will construct a homomorphism $\phi:\BG(n,\eps-2\del)\to G$. Because $\BG(n,\eps-2\del)$ has chromatic number at least $n+2$, it will follow that $\chi(G)\geq n+2$.

For each vertex $x\in S^n$, let $\phi(x)$ be the vertex of $G$ nearest to $x$. Note that $d(x,\phi(x))<\del$ for each $x\in S^n$, because $G$ is a $\del$-approximation. We check that $\phi$ maps edges of $\BG(n,\eps-2\del)$ to edges of $G$. Indeed, if $x,y\in S^n$ satisfy $d(x,-y)<\eps-2\del$ then
\[
d(\phi(x),-\phi(y))\leq d(\phi(x),x)+d(x,-y)+d(-y,-\phi(y))<\eps,
\]
so $\phi(x)$ and $\phi(y)$ are adjacent in $G$. It follows that $\phi$ is a well-defined homomorphism, and consequently that $\chi(G)\geq\chi(\BG(n,\eps-2\del))\geq n+2$.
\end{proof}

Our remaining two intermediate results require us to study walks in Borsuk graphs. Because these walks bounce between almost-antipodes, we introduce the following notation, which will allow us to instead reason about sequences of vertices approximating continuous paths in $S^n$.

\begin{dfn}\label{dfn:borsuk-walk}
Let $v_0,\ldots,v_k$ be vertices of a Borsuk graph $\BG(n,\eps)$ satisfying $d(v_i,v_{i+1})<\eps$ for all $0\leq i<k$. We write $[v_0,\ldots,v_k]$ to denote the walk $v_0(-v_1)v_2(-v_3)\cdots(\pm v_k)$ in $\BG(n,\eps)$.
\end{dfn}

\begin{rmk}
\Cref{dfn:borsuk-walk} partially motivates the requirement that $\del$-approximations have antipodally closed vertex sets.
Indeed, suppose $G$ is a $\del$-approximation to a Borsuk graph $\BG(n,\eps)$, and let $v_0,\ldots,v_k\in V(G)$ satisfy $d(v_i,v_{i+1})<\eps$ for all $i$. Because $V(G)$ is antipodally closed, the walk $[v_0,\ldots,v_k]$ of $\BG(n,\eps)$ is also a walk in $G$.
\end{rmk}

We now prove our remaining two intermediate results.

\begin{prop}\label{prop:approx-noninj}
Fix integers $r,n\geq 1$ and let $\eps=\frac{\pi}{2r+1}$. For any integer $N$ there exists some $\delta=\delta(r,n,N)>0$ such that if $G$ is a $\del$-approximation to $\BG(n,\eps)$ and $\phi:G\to H$ is a graph homomorphism with $|H|\leq N$ then $\phi$ is noninjective on some $C_{2r+3}$ of $G$.
\end{prop}

\begin{proof}
Let $\del$ be a small constant, to be determined, and set $\del_1=(4r+2)\del$ and $\del_2=\eps-2\del$. We first show that if $\del$ is sufficiently small then there are vertices $v,v'\in V(G)$ with $\phi(v)=\phi(v')$ and $\del_1\leq d(v,v')\leq 2\del_2$. Using this, we then construct a $C_{2r+3}$ in $G$ containing both $v$ and $v'$.

Choose some $x_0\in S^n$ and let $W$ be a maximal subset of $V(G)\cap B_{\del_2}(x_0)$ for which $d(x,x')\geq \del_1$ for all $x,x'\in W$. We claim that the family $\left\{B_{\del_1+\del}(x):x\in W\right\}$ covers $B_{\del_2-\del}(x_0)$. Suppose for the sake of contradiction that some $y\in B_{\del_2-\del}(x_0)$ had distance at least $\del_1+\del$ from each point of $W$. Then there would be a point $x\in V(G)\setminus W$ within distance $\del$ of $y$, so $x$ would be in $B_{\del_2}(x_0)$ and have distance at least $\del_1$ from each point of $W$. However, this contradicts the maximality of $W$.

We now have that $\left\{B_{\del_1+\del}(x):x\in W\right\}$ covers $B_{\del_2-\del}(x_0)$. Summing the measures of each set in the cover, we see that $\mu(n,\del_1+\del)\times|W|\geq\mu(n,\del_2-\del)$. As $\del\to 0$, causing $\del_1\to 0$ and $\del_2\to\eps$, the ratio $\mu(n,\del_2-\del)/\mu(n,\del_1+\del)$ goes to infinity. Thus, choosing $\del$ small enough yields
\[
N<\frac{\mu(n,\del_2-\del)}{\mu(n,\del_1+\del)}\leq|W|.
\]
By the pigeonhole principle, two vertices $v,v'\in W$ must satisfy $\phi(v)=\phi(v')$, as $H$ has only $N$ vertices. Moreover, we have $\del_1\leq d(v,v')\leq 2\del_2$ by definition of $W$.

We now construct a $C_{2r+3}$ in $G$ containing both $v$ and $v'$. Let $\theta=d(v,v')$. Let $\phi:S^1\to S^n$ be an isometric embedding of the unit circle in $S^n$ such that $\phi(0)=v$ and $\phi(-\theta)=v'$, where 0 and $-\theta$ refer to the points in $S^1$ with those radial angles. Define a sequence $x_0,\ldots,x_{2r+3}$ of points in $S^n$ with
\[
\begin{array}{c}
x_0=\phi(0),\ 
\ldots,\ 
x_i=\phi\left(\frac{i(\pi-\theta)}{2r+1}\right),\ 
\ldots,\ 
x_{2r+1}=\phi(\pi-\theta),\ 
\\[5pt]x_{2r+2}=\phi\left(\pi-\frac\theta 2\right),\ 
x_{2r+3}=\phi(\pi).
\end{array}
\]

Observe that for $0\leq i<2r+1$, we have
\[d(x_i,x_{i+1})=\frac{\pi-\theta}{2r+1}\leq\frac{\pi-\del_1}{2r+1}=\eps-2\del.\]
For $i=2r+1$ and $i=2r+2$, we also have
\[
d(x_i,x_{i+1})=\frac{\theta}2\leq\del_2=\eps-2\del.
\]
Let $v_i$ be the vertex of $G$ closest to $x_i$ for each $0\leq i\leq 2r+3$. Because $G$ is a $\del$-approximation, it follows that $d(v_i,v_{i+1})<2\del+d(x_i,x_{i+1})\leq\eps$ for all $i$. Therefore, $P=[v_0,v_1,\ldots,v_{2r+3}]$ is a walk in $G$.

We observe that $x_0=-x_{2r+3}=v$ and that $x_{2r+1}=-v'$. It follows that $v_i=x_i$ at these three indices, so $P$ has the form
\[
[v_0,v_1,\ldots,v_{2r+3}]=v_0(-v_1)\cdots (-v_{2r+1})v_{2r+2}(-v_{2r+3})=v(-v_1)\cdots v'v_{2r+2}v.
\]
In particular, $P$ is a closed walk of length $2r+3$ containing both $v$ and $v'$. Because $\BG\left(n,\frac{\pi}{2r+1}\right)$ contains no odd cycles shorter than $C_{2r+3}$, $P$ cannot repeat any vertices. Thus, this is a $C_{2r+3}$ on which $\phi$ acts noninjectively.
\end{proof}


\begin{prop}\label{prop:approx-simp-conn}
Fix any $n\geq 2$ and $0<\del<\eps/3$. If $G$ is a $\del$-approximation to $\BG(n,\eps)$ then $G$ is simply connected.
\end{prop}

\begin{proof}
Let $P=[v_0,\ldots,v_k]$ and $P'=[v'_0,\ldots,v'_\ell]$ be walks in $G$ with $k\equiv \ell\pmod 2$ and additionally satisfying $v_0=v'_0$ and $v_k=v'_\ell$. To show that $G$ is simply connected, we must show that $P$ and $P'$ are homotopic in $G$. Our proof extends the discrete walks $P$ and $P'$ to continuous paths $f$ and $f'$ in $S^n$, takes a continuous homotopy $F$ between $f$ and $f'$, and lastly discretizes $F$ to prove a homotopy equivalence between $P$ and $P'$.

Let $f:[0,1]\to S^n$ be a continuous function with $f(\frac{i}k)=v_i$ for each $0\leq i\leq k$, such that $f$ restricts to a geodesic on each interval $\left[\frac ik,\frac{i+1}k\right]$. Analogously, let $f':[0,1]\to S^n$ be a continuous extension of $P'$. Because $S^n$ is simply connected when $n\geq 2$, there is a continuous homotopy $F:[0,1]^2\to S^n$ between $f$ and $f'$. That is, $F$ is a continuous function with boundary conditions $F(0,-)=f$ and $F(1,-)=f'$, which also satisfies $F(x,0)=v_0$ and $F(x,1)=v_k$ for all $x\in[0,1]$.

Because $F$ is a continuous function on a compact domain, $F$ is uniformly continuous. Thus, there is $\gamma>0$ such that if $x,x',y,y'\in[0,1]$ satisfy $|x-x'|<\gamma$ and $|y-y'|<\gamma$ then $d(F(x,y),F(x',y'))<\del$. Choose sequences $0=x_0\leq x_1\leq\dots\leq x_s=1$ and $0=y_0\leq y_1\leq\dots\leq y_t=1$ such that $x_{i+1}-x_i<\gamma$ for all $0\leq i<s$ and $y_{j+1}-y_j<\gamma$ for all $0\leq j<t$. We additionally require that each of the fractions $\frac 1k,\ldots,\frac{k-1}k$ and $\frac 1\ell,\ldots,\frac{\ell-1}\ell$ occurs at least twice in the sequence $y_0,\ldots,y_t$ and that $t\equiv k\equiv\ell\pmod 2$; these conditions may be achieved by adding more elements to the sequence $y_0,\ldots,y_t$.

We now present a homotopy equivalence between $P$ and $P'$. We construct walks $Q_0,\ldots,Q_s$ of $G$ such that $Q_i$ is a discrete approximation of the continuous path $F(x_i,-)$. We  then show that $Q_i$ is homotopic to $Q_{i+1}$ for all $0\leq i<s$, and that $Q_0$ is homotopic to $P$ and $Q_s$ is homotopic to $P'$.

For all $0\leq i\leq s$ and $0\leq j\leq t$, define $w_{i,j}$ to be the vertex of $G$ closest to $F(x_i,y_j)$. Observe that $w_{i,0}=v_0$ and $w_{i,t}=v_k$ for all $i$. Moreover, if $|i_1-i_2|,|j_1-j_2|\leq 1$ then $|x_{i_1}-x_{i_2}|,|y_{j_1}-y_{j_2}|<\gamma$, so
\[
d\left(w_{i_1,j_1},w_{i_2,j_2}\right)<d\left(w_{i_1,j_1},F(x_{i_1},y_{j_1})\right)+d\left(F(x_{i_1},y_{j_1}),F(x_{i_2},y_{j_2})\right)+d\left(F(x_{i_2},y_{j_2},w_{i_2,j_2})\right)< 3\del<\eps.
\]
For $0\leq i\leq s$, let $Q_i=[w_{i,0},w_{i,1},\ldots,w_{i,t}]$. The prior inequality implies that $Q_i$ is in fact a walk of $G$.

We show that $Q_i$ is homotopic to $Q_{i+1}$ via the following sequence of substitution steps. 
\begin{align*}
Q_i=&\ [w_{i,0},w_{i,1},\ldots,w_{i,t}]=[w_{i+1,0},w_{i,1},\ldots,w_{i,t}]
\subarrow[w_{i+1,0},w_{i+1,1},w_{i,2},\ldots,w_{i,t}]
\\&\subarrow\cdots\subarrow [w_{i+1,0},\ldots,w_{i+1,j},w_{i,j+1},\ldots,w_{i,t}]\subarrow\cdots
\\&\subarrow [w_{i+1,0},\ldots,w_{i+1,t-1},w_{i,t}]=[w_{i+1,0},\ldots,w_{i+1,t}]=Q_{i+1}.
\end{align*}
At the $j^\text{th}$ step, we replace $w_{i,j}$ with $w_{i+1,j}$. This is a valid substitution because we have shown that both $w_{i,j}$ and $w_{i+1,j}$ are within $\eps$ of the vertices $w_{i+1,j-1}$ and $w_{i,j+1}$ on either side of the substitution.

Thus, it only remains to show that $Q_0$ is homotopic to $P$ --- which, by symmetry, also shows that $Q_s$ is homotopic to $P'$.
Write $w_j$ for $w_{0,j}$. We claim there are indices $0=j_0<j_1<\cdots<j_k=t$ such that, for all $0\leq i\leq k$, we have $j_i\equiv i\pmod 2$ and $w_{j_i}=v_i$. For $0<i<k$, such a $j_i$ exists because of the constraint that $y_j=y_{j+1}=\frac ik$ for some index $j$, allowing us to choose $j_i\in\{j,j+1\}$ of the appropriate parity. Additionally, $j_k=t$ has the correct parity because our construction stipulates that $t\equiv k\pmod 2$.

We claim that, for each $i$, the walks $[v_i,v_{i+1}]$ and $[w_{j_i},w_{j_i+1},\ldots,w_{j_{i+1}}]$ are homotopy equivalent. 
For any index $j_i\leq h\leq j_{i+1}$, notice that
\begin{align*}
d(v_i,w_h)
&\leq d(v_i,f(y_h)) + d(f(y_h),w_h)
\leq d(v_i,f(y_h)) + d(f(y_h),v_{i+1}) 
=d(v_i,v_{i+1})<\eps.
\end{align*}
The second inequality follows because $w_h$ is the vertex of $G$ nearest to $f(y_h)$, and the subsequent equality holds because $f(y_h)$ is on a geodesic between $v_i$ and $v_{i+1}$. Hence, we may perform the following sequence of substitutions on $[w_{j_i},w_{j_i+1},\ldots,w_{j_{i+1}}]$.
\begin{align*}
&[w_{j_i},w_{j_i+1},\ldots,w_{j_{i+1}}]=[v_i,w_{j_i+1},w_{j_i+2},\ldots,w_{j_{i+1}}]
\subarrow [v_i,v_i,w_{j_i+2},\ldots,w_{j_{i+1}}]
\\&\subarrow [v_i,v_i,v_i,w_{j_i+3},\ldots,w_{j_{i+1}}]
\subarrow\cdots\subarrow[v_i,\ldots,v_i,w_{j_{i+1}}]=[v_i,\ldots,v_i,v_{i+1}]
\end{align*}
Repeated deletions then transform this walk (which has odd length $j_{i+1}-j_i$) into the single edge $[v_i,v_{i+1}]$, thus showing that $[w_{j_i},w_{j_i+1},\ldots,w_{j_{i+1}}]$ and $[v_i,v_{i+1}]$ are homotopic walks. One may show that their antipodal images $[-v_i,-v_{i+1}]$ and $[-w_{j_i},-w_{j_i+1},\ldots,-w_{j_{i+1}}]$ are homotopic by performing an analogous sequence of substitution and deletion operations.

Observe that $P$ may be written as a concatenation of edges $(-1)^i[v_i,v_{i+1}]=[(-1)^iv_i,(-1)^iv_{i+1}]$ and $Q_0$ as a concatenation of subwalks $(-1)^i[w_{j_i},\ldots,w_{j_{i+1}}]$. The $k$ terms in each concatenation are pairwise homotopic, yielding a homotopy equivalence between $P$ and $Q_0$. An analogous argument shows $Q_s$ and $P'$ are homotopic, completing the proof that $P$ and $P'$ are homotopic.
\end{proof}

Combining \cref{prop:approx-chi,prop:approx-noninj,prop:approx-simp-conn} with \cref{mainthm} lets us lower-bound $\dhom(C_{2r+1})$.

\begin{theorem}\label{thm:approx-seq}
Fix $r\geq 2$, and set $n=8r^2-2$ and $\eps=\frac{\pi}{2r+1}$. Let $\eps/3>\delta_1>\delta_2>\ldots$ be a sequence monotonically decreasing to 0 and let $G_1,G_2,\ldots$ be a sequence of graphs such that $G_i$ is a $\delta_i$-approximation to $\BG(n,\eps)$. Then the family $G_1,G_2,\ldots$ has no finite $C_{2r+1}$-free homomorphic image.
\end{theorem}

\begin{proof}
Suppose on the contrary there were a finite $C_{2r+1}$-free graph $H$ together with graph homomorphisms $\phi_k:G_k\to H$ for each $k$. Applying \cref{prop:approx-noninj} with $N=|H|$, there is $k$ sufficiently large that $\phi_k$ is noninjective on some $C_{2r+3}$ of $G_k$. Moreover, $G_k$ is simply connected by \cref{prop:approx-simp-conn}.

Applying \cref{mainthm} to $\phi_k:G_k\to H$, we see that $\chi(G_k)<8r^2$. However, $\chi(G_k)\geq n+2=8r^2$ by \cref{prop:approx-chi}, yielding the desired contradiction.
\end{proof}

Let $r$, $n$, and $\eps$ be as in \cref{thm:approx-seq}. To prove \cref{thm:dhom>0}, we apply \cref{thm:approx-seq} to a sequence of random graphs approximating $\BG(n,\eps)$. This lower-bounds $\dhom(C_{2r+1})$ by $\mu(\eps,n)$, the fraction of $S^n$ covered by a ball of radius $\eps$.

\begin{corollary}\label{cor:bound-dhom}
Fix $r\geq 2$, and set $n=8r^2-2$ and $\eps=\frac{\pi}{2r+1}$. Then $\dhom(C_{2r+1})\geq\mu(n,\eps)$.
\end{corollary}

\begin{proof}
Fix a sequence $\eps/4>\del_1>\del_2>\cdots$ monotonically decreasing to zero. We construct $\delta_k$-approximations $G_k$ to $\BG(n,\eps)$ such that $\del(G_k)/|G_k|$ approaches $\mu(n,\eps)$.

Let $N$ be a sufficiently large integer. Choose $N$ points $x_1,\ldots,x_N\in S^n$ uniformly at random, and let $G_k$ be the subgraph of $\BG(n,\eps)$ induced by the vertex set $V(G_k)=\{x_1,-x_1,\ldots,x_N,-x_N\}$. With high probability as $N\to\infty$, the graph $G_k$ will be a $\del_k$-approximation to $\BG(n,\eps)$. Furthermore, any vertex $v\in V(G)$ has expected degree $1+\mu(n,\eps)(2N-2)$. A Chernoff bound implies that, with high probability as $N\to\infty$, every vertex $v\in V(G_k)$ satisfies
\[\deg(v)=(2+o(1))\mu(n,\eps)\cdot N=(1+o(1))\mu(n,\eps)\cdot|G_k|.\]
Thus, taking $N$ large enough, this construction yields a $\del_k$-approximation $G_k$ to $\BG(n,\eps)$ with minimum degree
\[
\del(G_k)\geq(1-\del_k)\mu(n,\eps)\cdot|G_k|.
\]

By \cref{thm:approx-seq}, the family $G_1,G_2,\ldots$ has no finite $C_{2r+1}$-free homomorphic image. Thus,
\[
\dhom(C_{2r+1})\geq\limsup_{k\to\infty}\frac{\del(G_k)}{|G_k|}\geq\mu(n,\eps).\qedhere
\]
\end{proof}


\begin{rmk}
The lower bound $\mu\left(8r^2-2,\frac{\pi}{2r+1}\right)$ obtained in \cref{cor:bound-dhom} is $r^{-(8+o(1))r^2}$ as $r\to\infty$.
\end{rmk}

\section{The Discrete Fundamental Group}\label{s:topology}

Returning to our homotopy theory, we introduce a graph-theoretic analogue of the fundamental group. Using the discrete fundamental group, we develop an unexpected relationship between our homotopy theory and Lov\'asz's neighborhood complex.

We briefly recall the definition of the fundamental group in algebraic topology. Let $X$ be a path-connected topological space. A \emph{loop} based at a point $x_0\in X$ is a continuous map $L:[0,1]\to X$ with $L(0)=L(1)=x_0$. Two such loops are \emph{homotopic} if one loop may be continuously deformed into the other while keeping the basepoint $x_0$ fixed. Let $\pi_1(X,x_0)$ be the set of homotopy classes of loops in $X$ based at $x_0$. It is well known (see \cite[\S 1.1]{Ha02}) that $\pi_1(X,x_0)$ forms a group under concatenation, and that $\pi_1(X,x_0)\simeq\pi_1(X,x_1)$ for any two points $x_0,x_1\in X$. Due to this isomorphism, the fundamental group is often written simply as $\pi_1(X)$.

Let $G$ be a connected graph. Given a vertex $v_0\in V(G)$, the \emph{discrete fundamental group} $\pi_1(G,v_0)$ consists of those closed walks in $G$ with both endpoints at $v_0$, up to homotopy equivalence. This forms a group under concatenation, where the inverse of a walk $v_0v_1\cdots v_{k-1}v_0$ is given by $v_0v_{k-1}\cdots v_1v_0$. Let $\pieven(G,v_0)$ be the subgroup of $\pi_1(G,v_0)$ comprising closed walks of even length.

We state without proof two results which are proved analogously to elementary statements in algebraic algebraic topology.

\begin{prop}[See Proposition 1.5 in \cite{Ha02}]\label{prop:fg-basepoint-independent}
	Let $G$ be a connected graph and let $v_0$ and $v_1$ be vertices of $G$. There is an isomorphism from $\pi_1(G,v_0)$ to $\pi_1(G,v_1)$. Moreover, this isomorphism maps $\pieven(G,v_0)$ to $\pieven(G,v_1)$.
\end{prop}

We remark that the proof of \cref{prop:fg-basepoint-independent} is also very similar to that of \cref{cor:homt-loops}.
This proposition allows us to simply notate the discrete fundamental group as $\pi_1(G)$, and its subgroup of even walks as $\pieven(G)$. 

\begin{prop}[See Proposition 1.6 in \cite{Ha02}]\label{prop:disc-simp-conn}
	Let $G$ be a connected graph. $G$ is simply connected if and only if $\pieven(G)$ is trivial.
\end{prop}

We conclude this section by demonstrating a connection between the discrete fundamental group and Lov\'asz's neighborhood complex. 
Given a graph $G$, its \emph{neighborhood complex} $\NN(G)$ is a simplicial complex on $V(G)$, with faces corresponding to those subsets $W\subseteq V(G)$ which have a common neighbor. For example, the neighborhood complex of the complete graph $K_n$ has a face corresponding to each proper subset of $[n]$, so $\NN(K_n)$ is the boundary of an $(n-1)$-dimensional simplex.

Notice that $\NN(G)$ is a connected topological space if and only if $G$ is connected and non-bipartite. In this regime, the fundamental group of $\NN(G)$ is in fact isomorphic to $\pieven(G)$. As a corollary, $G$ is a simply connected graph if and only if $\NN(G)$ is a simply connected topological space, i.e.\ if and only if $\pi_1(\NN(G))$ is trivial.

\begin{prop}\label{prop:pi1-NG}
Let $G$ be a connected non-bipartite graph. Then $\pi_1(\NN(G))\simeq \pieven(G)$.
\end{prop}

The proof of \cref{prop:pi1-NG} uses the \emph{edge-path group}, an alternate representation of the fundamental group of a simplicial complex (see e.g.\ \cite[\S6.6]{Ad16}). Let $X$ be a simplicial complex on vertex set $V$. An \emph{edge path} of $X$ is a sequence $v_0\cdots v_k$ of vertices such that $v_i$ and $v_{i+1}$ are contained in the same simplex of $X$ for all $i<k$. We say two edge paths are \emph{equivalent} if we may transform one into the other via a finite sequence of steps of the following form:
\begin{homtenum}
\item[(Ins1)] Given an edge path $v_0\cdots v_k$ and an index $0\leq i\leq k$, replace $v_i$ with $v_iv_i$.
\item[(Del1)] Given an edge path $v_0\cdots v_k$ and an index $0\leq i<k$ satisfying $v_i=v_{i+1}$, replace $v_iv_{i+1}$ with $v_i$.
\item[(Ins2)] Given an edge path $v_0\cdots v_k$, an index $0\leq i<k$, and a vertex $v'$ such that $v_i,v',v_{i+1}$ are contained in the same simplex of $X$, replace $v_iv_{i+1}$ with $v_iv'v_{i+1}$.
\item[(Del2)] Given an edge path $v_0\cdots v_k$ and an index $0<i<k$ such that $v_i, v_{i+1}, v_{i+2}$ are contained in the same simplex of $X$, replace $v_iv_{i+1}v_{i+2}$ with $v_iv_{i+2}$.
\end{homtenum}
Given a vertex $v_0$ of $X$, the \emph{edge-path group} $\EE(X,v_0)$ consists of equivalence classes of edge paths starting and ending at $v_0$, with concatenation as the group operation. It is well known that $\EE(X,v_0)$ is isomorphic to $\pi_1(X)$. We prove \cref{prop:pi1-NG} by constructing an isomorphism between $\EE(\NN(G),v_0)$ and $\pieven(G,v_0)$ for some vertex $v_0\in V(G)$.

\newcommand{\pphi}{\overline{\phi}}
\begin{proof}[Proof of \cref{prop:pi1-NG}]
We define a map $\phi$ from \{even closed walks of $G$ based at $v_0$\} to \{edge paths in $\NN(G)$ based at $v_0$\}. We show that $\phi$ projects to a well-defined group homomorphism $\pphi:\pieven(G,v_0)\to\EE(\NN(G),v_0)$, and that $\pphi$ is an isomorphism.

Given an even closed walk $P=v_0\cdots v_{2k-1}v_0$, let $\phi(P)$ be the edge path $v_0v_2v_4\cdots v_{2k-2}v_0$. We show that $\phi$ maps homotopic closed walks to equivalent edge paths, by checking that the equivalence class of $\phi(P)$ is preserved by the three homotopy operations described in \cref{dfn:homotopy}.
\begin{homtenum}
\item[(Ins)]Let $P'=v_0\cdots v_iwv_i\cdots v_{2k-1}v_0$ be formed from $P$ by a single insertion. If $i$ is even, then $\phi(P')$ is obtained by replacing $v_i$ by $v_iv_i$ in $\phi(P)$, as in (Ins1). If $i$ is odd then $\phi(P')$ is obtained by replacing $v_{i-1}v_{i+1}$ by $v_{i-1}wv_{i+1}$ in $\phi(P)$. The vertices $v_{i-1}$, $w$, and $v_{i+1}$ have a common neighbor $v_i$, so they form a simplex of $\NN(G)$. Thus, $\phi(P')$ is equivalent to $\phi(P)$ by (Ins2).
\item[(Del)] If $P'$ is formed from $P$ by a single deletion step, then $P$ is formed from $P'$ by a single insertion step. By the above argument, $\phi(P)$ is equivalent to $\phi(P')$.
\item[(Sub)]Let $P'=v_0\cdots v_{i-1}v'_iv_{i+1}\cdots v_{2k-1}v_0$ be formed from $P$ by a single substitution. If $i$ is odd then $\phi(P')=\phi(P)$. If $i$ is even then $\phi(P')$ is formed by replacing $v_{i-2}v_iv_{i+2}$ by $v_{i-2}v'_iv_{i+2}$ in $\phi(P)$. Then $\phi(P)$ and $\phi(P')$ are equivalent by the two steps
\begin{align*}
\phi(P)=v_0\cdots v_{i-2}v_iv_{i+2}\cdots v_0
&\homtarrow{Ins2}
v_0\cdots v_{i-2}v'_iv_iv_{i+2}\cdots v_0
\\&\homtarrow{Del2}
v_0\cdots v_{i-2}v'_iv_{i+2}\cdots v_0=\phi(P'),
\end{align*}
which are permissible because $\{v_{i-2},v'_i,v_i\}$ and $\{v'_i,v_i,v_{i+2}\}$ are simplices of $\NN(G)$ with common neighbors $v_{i-1}$ and $v_{i+1}$, respectively.
\end{homtenum}
It follows that $\phi$ maps homotopic closed walks to equivalent edge paths, yielding a well-defined group homomorphism $\pphi:\pieven(G,v_0)\to\EE(\NN(G),v_0)$.

We claim $\pphi$ is surjective. Let $Q=v_0v_2v_4\cdots v_{2k-2}v_0$ be an edge path of $G$. Because $v_{2i}$ and $v_{2i+2}$ are contained in the same simplex of $\NN(G)$, they share a common neighbor $v_{2i+1}$. Thus, there is a closed walk $P=v_0v_1\cdots v_{2k-1}v_0$ with $\phi(P)=Q$.

Lastly, we prove that $\pphi$ is injective. We show that if $P$ and $P'$ are closed walks with $\phi(P)=Q$ equivalent to $\phi(P')=Q'$, then $P$ and $P'$ are homotopic. It suffices to consider the cases that $Q$ and $Q'$ are related by (Ins1) or (Ins2), as (Del1) and (Del2) are the respective inverse transformations. Write $Q=v_0\cdots v_{k-1}v_0$ and $P=v_0w_0\cdots v_{k-1}w_{k-1}v_0$.
\begin{homtenum}
\item[(Ins1)] Suppose $Q'=v_0\cdots v_iv_i\cdots v_{k-1}v_0$ is formed from $Q$ by (Ins1). Then $P'$ takes the form
\[P'=v_0w'_0v_1w'_1\cdots v_iw'_iv_iw''_iv_{i+1}w'_{i+1}\cdots v_{k-1}w'_{k-1}v_0.\]
Observe that $P$ is homotopic to the closed walk
\[
P''=v_0w'_0v_1w'_1\cdots v_iw'_iv_{i+1}w'_{i+1}\cdots v_{k-1}w'_{k-1}v_0
\]
via a sequence of substitutions and that $P'$ may be obtained from $P''$ by a single insertion. Thus, $P$ and $P'$ are homotopic closed walks.

\item[(Ins2)] Suppose $Q'=v_0\cdots v_iv'v_{i+1}\cdots v_{k-1}v_0$ is formed from $Q$ by (Ins2). Then $P'$ takes the form
\[P'=v_0w'_0v_1w'_1\cdots v_iw'_iv'w''_iv_{i+1}w'_{i+1}\cdots v_{k-1}w'_{k-1}v_0.\]
Because $v_i,v',v_{i+1}$ are in the same simplex of $\NN(G)$, and thus share a common neighbor $x$. By sequences of substitutions, $P$ is homotopic to the closed walk
\[
P''=v_0w'_0v_1w'_1\cdots v_ixv_{i+1}w'_{i+1}\cdots v_{k-1}w'_{k-1}v_0.
\]
and $P'$ is homotopic to
\[
P'''=v_0w'_0v_1w'_1\cdots v_ixv'xv_{i+1}w'_{i+1}\cdots v_{k-1}w'_{k-1}v_0.
\]
Because $P'''$ is obtained from $P''$ via a single deletion, the closed walks $P$ and $P'$ are homotopic.
\end{homtenum}
It follows that any closed walks $P$ and $P'$ with $\pphi(P)=\pphi(P')$ are homotopic. Hence, $\phi$ is injective, completing the proof that $\pphi$ is a group isomorphism.
\end{proof}

Combining \cref{prop:disc-simp-conn,prop:pi1-NG} yields another characterization of simply connected graphs which is equivalent to \cref{dfn:simp-conn} for non-bipartite graphs.

\begin{corollary}
	Suppose $G$ is a connected non-bipartite graph. Then, $G$ is simply connected if and only if $\pi_1(\NN(G))$ is trivial.
\end{corollary}

\section{Concluding Remarks}\label{s:conclusion}

We discuss possible extensions of \cref{mainthm}, as well as its limitations when used to lower-bound $\dhom(C_{2r+1})$. Lastly, we conclude with some open questions of a more topological flavor.

\subsection{Bounding $\dhom(C_{2r+1})$}

As mentioned in \cref{rmk:mainthm-tight}, our proof of \cref{mainthm} did not attempt to optimize the constant coefficient in the bound on $\chi(G)$. We conjecture that under the hypotheses of \cref{mainthm}, $\chi(G)$ is in fact at most $2r$.

\begin{conjecture}\label{conj:mainthm-strong}
	Let $G$ be a simply connected graph and fix $r\geq 2$. Suppose $\phi:G\to H$ is a graph homomorphism such that $H$ is $C_{2r+1}$-free, and suppose further that some odd cycle $C$ of $G$ satisfies $|\phi(V(C))|\leq 2r+2$. Then $\chi(G)\leq 2r$.
\end{conjecture}

We cannot guarantee $\chi(G)<2r$, as shown by the following construction. Let $G$ be a large simply-connected $C_{2r+1}$-free graph with $\chi(G)=2r$. Such a graph may be obtained by taking a sufficiently good approximation to a Borsuk graph $\BG\left(2r-2,\eps\right)$, with $\eps\leq\frac{\pi}{2r+1}$ chosen small enough that the graph $G$ has chromatic number exactly $2r$. The graph $G$, together with any homomorphism $\phi:G\to K_{2r}$, satisfies the hypotheses of \cref{mainthm}.

It would be interesting to understand how $\dhom(C_{2r+1})$ grows as a function of $r$. Even if we could show \cref{conj:mainthm-strong}, this would imply a lower bound of $\dhom(C_{2r+1})\geq r^{-(2+o(1))r}$, which remains far from the upper bound $\dhom(C_{2r+1})\leq\frac 1{2r+1}$ shown by Ebsen and Schacht \cite{EbSc20}. One specific question in this direction is to determine whether the best possible lower bound on $\dhom(C_{2r+1})$ is given by a sequence of approximations to a Borsuk graph, i.e.\ whether the exact value of $\dhom(C_{2r+1})$ takes a form similar to the lower bound of \cref{cor:bound-dhom}.
\begin{question}
	Fix $r\geq 2$ and set $\eps=\frac{\pi}{2r+1}$. Is $\dhom(C_{2r+1})=\mu(n,\eps)$ for some $n$? If so, is $n=2r-1$?
\end{question}

\subsection{Topological Open Problems}

We remark that \cref{mainthm} may be generalized from simply connected graphs to any graph $G$ with cyclic discrete fundamental group.

\begin{theorem}\label{mainthm-gen}
	Let $G$ be a graph with $\pi_1(G)$ cyclic and fix $r\geq 2$. Suppose $\phi:G\to H$ is a graph homomorphism such that $H$ is $C_{2r+1}$-free, and suppose further that some odd cycle $C$ of $G$ satisfies $|\phi(V(C))|\leq 2r+2$. Then $\chi(G)<8r^2$. 
\end{theorem}

\begin{proof}
We follow the proof of \cref{mainthm}, with one exception. The condition that $G$ is simply connected is used exactly once, to show that $\I(C)=\I(C'')$ within the proof of \cref{cor:homt-loops}. We claim that this step still holds under the alternate hypothesis that $\pi_1(G)$ is cyclic.

Assume the proof setup of \cref{cor:homt-loops}. Let $P$ be a closed walk generating $\pi_1(G,v_0)$. Because $C$ is an odd closed walk based at $v_0$, there is $a\in\mathbb Z$ such that $C$ is homotopic to $P^a$, the closed walk formed by concatenating $a$ copies of $P$. Furthermore, because $C$ has odd length and walk parities are preserved by homotopy equivalence, it follows that $a$ must be odd. Hence, $\I(C)=\I\left(P^a\right)=a\times\I(P)=\I(P)\pmod 2$. Similarly, $\I(C'')=\I(P)$, yielding that $\I(C)=\I(C'')$. It follows that \cref{cor:homt-loops} holds in the more general case that $\pi_1(G)$ is cyclic. Thus, the proof of \cref{mainthm} generalizes to all graphs $G$ with $\pi_1(G)$ cyclic.	
\end{proof}

Because of the interplay between the discrete fundamental group and chromatic number in \cref{mainthm-gen}, it is natural to ask about thresholds with a more topological nature. Let $G$ be an $H$-free graph on $n$ vertices, with minimum degree $\del(G)\geq\al n$ for some fixed $\al$. It would be informative to understand the relationship between $G$ and various topological properties of $G$, such as the structure of $\NN(G)$ or $\pi(G)$. Similar topological properties have been studied for random models of graphs and hypergraphs \cite{ArLiLuMe13,Ka07,LiPe16}, but it is not clear how these properties would interact with the $H$-free condition. We ask about one such ``topological threshold'' explicitly, but any results of this type would be quite interesting.

\begin{question}
	Fix a graph $H$. Let $\dcyc(H)$ be the infimum of those $\al$ such that if $G$ is an $H$-free graph with $\delta(G)>\al|G|$, then $\pi_1(G)$ is cyclic. What is the value of $\dcyc(H)$?
\end{question}

At first glance, it seems more intuitive to ask for the least minimum bound on $\delta(G)$ which implies that $G$ is simply connected. However, being simply connected can be too restrictive a property. For example, it is well-known that triangle-free graphs $G$ with $\delta(G)>\frac 25|G|$ are bipartite. Conversely, if $G$ is a balanced blowup of $C_5$, then $\delta(G)=\frac 25|G|$ and $\pi_1(G)\simeq\mathbb Z$. Thus, determining the threshold $\al$ above which triangle-free graphs $G$ with $\del(G)>\al|G|$ are simply connected is fundamentally a problem about bipartite --- rather than triangle-free --- graphs. Moreover, the generalization of \cref{mainthm} to \cref{mainthm-gen} suggests that having cyclic discrete fundamental group is the more germane property.

Looking beyond the fundamental group, it would be valuable to formulate discrete analogues of higher homotopy groups or homology groups. Such constructions would be interesting in their own right, especially if they coincided on a subgroup with the corresponding homology or homotopy group of $\NN(G)$. It would also be illuminating to understand which properties of a graph $G$ --- such as its chromatic number --- could be determined by analyzing these groups.

\paragraph{Acknowledgements.} The author thanks Jacob Fox and Yuval Wigderson for helpful comments.

\end{document}